\documentclass{amsart}
\usepackage{amsmath,amsthm}
\usepackage{amsfonts,amssymb}

\usepackage{enumerate}
\usepackage{graphicx}
\usepackage{todonotes}

\newtheorem{thm}{Theorem}[section]

\newtheorem{lem}[thm]{Lemma}
\newtheorem{prop}[thm]{Proposition}
\newtheorem{exam}[thm]{Example}

\newtheorem{defn}[thm]{Definition}
\newtheorem{conj}[thm]{Conjecture}

\theoremstyle{remark}
\newtheorem{rem}[thm]{Remark}

\newtheorem*{exam*}{Example}

\def\bell{{\boldsymbol{\large {\ell}}}} 
\def\bnu{{\boldsymbol{\large {\nu}}}} 
 
\def\bka{{\boldsymbol{\large {\kappa}}}}

 \def\a{{\alpha}}
 \def\b{{\beta}}
 \def\g{{\gamma}}
 \def\k{{\kappa}}
 
 \def\l{{\lambda}}

 \def\la{{\langle}}
 \def\ra{{\rangle}}

 \def\sh{{\mathsf h}}
 
 \def\sB{{\mathsf B}}

 \def\sH{{\mathsf H}}

 \def\sP{{\mathsf P}}
 \def\sQ{{\mathsf Q}}
 \def\sR{{\mathsf R}}

 \def\xb{{\mathbf x}}
 \def\yb{{\mathbf y}}

 \def\CH{{\mathcal H}}
 \def\CI{{\mathcal I}}
 
 \def\CL{{\mathcal L}}
 \def\CE{{\mathcal E}}

 \def\CV{{\mathcal V}}

 \def\NN{{\mathbb N}}
 
 \def\QQ{{\mathbb Q}}
 \def\RR{{\mathbb R}}
 
 \def\ZZ{{\mathbb Z}}

\def\dim{\operatorname{dim}}

\newcommand{\wh}{\widehat}

\def\f{\frac}
\def\one{{\mathbf{1}}}

\def\xt{\tilde x}

\DeclareMathAlphabet{\mathpzc}{OT1}{pzc}{m}{it}
\def\za{\mathpzc{a}}

\begin{document}

\title[ ]
{Hahn polynomials for hypergeometric distribution}

\date{December 21, 2020}

\author{Plamen~Iliev}
\address{P.~Iliev, School of Mathematics, Georgia Institute of Technology, 
Atlanta, GA 30332--0160, USA}
\email{iliev@math.gatech.edu}

\author{Yuan~Xu}
\address{Y.~Xu, Department of Mathematics, University of Oregon, Eugene, 
OR 97403--1222, USA}
\email{yuan@uoregon.edu}
\thanks{The first author is partially supported by Simons Foundation Grant \#635462.}

\keywords{Hypergeometric distribution, Hahn polynomials, several variables, factorizations}
\subjclass[2010]{33C50, 33C70, 42C05}

\begin{abstract}
Orthogonal polynomials for the multivariate hypergeometric distribution are defined on lattices in polyhedral
domains in $\RR^d$. Their structures are studied through a detailed analysis of classical Hahn polynomials with
negative integer parameters. Factorization of the Hahn polynomials is explored and used to explain 
the relation between the index set of orthogonal polynomials and the lattice set in polyhedral domain. In the multivariate case, these constructions lead to nontrivial families of hypergeometric polynomials vanishing on lattice polyhedra. The generating functions and bispectral properties of the orthogonal polynomials are also discussed. 
\end{abstract}
 \maketitle

\section{Introduction}
\setcounter{equation}{0}

We study orthogonal polynomials with respect to the hypergeometric distribution in several variables. Let
$N$ be a positive integer and let $\ell_i$, $1\le i \le d+1$, be nonnegative integers, such that $\ell_i \le N$ 
and $\ell_i+\ell_j \ge N$ for $i \ne j$. The hypergeometric 
distribution in $d$ variables is defined by 
\begin{equation}\label{eq:poly-weight}
  \sH_{\ell,N}(x) : =  \frac{N!}{(-|\ell |)_N}\prod_{i=1}^d \frac{(-\ell_i)_{x_i}}{x_i!} \frac{(-\ell_{d+1})_{N-|x|}}{(N-|x|)!}, 
\end{equation}
where $|\ell| = \ell_1+\cdots+ \ell_{d+1}$, $|x|=x_1+\cdots +x_d$ and $(a)_k = a(a+1)\ldots(a+k-1)$ is the
Pochhammer symbol; it is a probability distribution on the set $V_{\ell,N}$ of discrete polyhedral domain
defined by 
\begin{equation}\label{eq:poly-domain}
   V_{\ell, N}^d: = \left\{x \in \NN_0^d: 0 \le x_i \le \ell_i, \, 1 \le i \le d, \, \hbox{and}\, N-\ell_{d+1} \le |x| \le N\right\}.
\end{equation}
We studied orthogonal polynomials with respect to this distribution in \cite{IX20}, which are called the 
Hahn polynomials on polyhedra, since an orthogonal basis can be explicitly given in terms of the Hahn 
polynomials. These polynomials are closely related to the classical Hahn polynomials of several variables 
\cite{KM} that are orthogonal with respect to the weight function 
\begin{equation}\label{eq:HK-weight}
 W_{\k,N} (x) = \frac{N!}{(|\k|+ d+1)_N} \prod_{i=1}^d   \frac{(\k_i + 1)_{x_i}}{x_i!} \frac{(\k_{d+1}+1)_{N-|x|}}{(N-|x|)!}  
\end{equation}
where $\k_i > -1$, $1\le i \le d+1$, defined on the lattice points within the simplex $V_N^d$ defined by
$V_N^d = \{x \in \NN_0^d: |x| \le N\}$.

In \cite{IX20} a family of Hahn polynomials on polyhedra, denoted by $\sQ_\nu(\cdot ; \ell ,N)$, are explicitly
given with $\nu$ belonging to an index set $H_{\ell,N}^d$, which is a subset of $V_N^d$, but different from $V_{\ell, N}^d$. The set 
$H_{\ell,N}^d$ has a complicated structure and one of the main results of \cite{IX20} is to show, with a strenuous 
combinatorial proof, that $H_{\ell,N}^d$ and $V_{\ell,N}^d$ have the same cardinality, so that 
$\{\sQ_\nu(\cdot ; \ell ,N): \nu \in H_{\ell,N}^d\}$ is an orthogonal basis. The definition of $H_{\ell,N}^d$ 
comes from setting $\k_i$ as negative integers $- \ell_i -1$ in a basis of classical Hahn polynomials 
and collecting those polynomials whose norms are finite and non-zero. Since the basis used in \cite{IX20} 
is specifically normalized, one may ask if the set $H_{\ell,N}^d$ is uniquely determined. Furthermore, 
since the set $H_{\ell,N}^d$ is a subset of $V_N^d$, one may ask if those classical Hahn polynomials 
whose indices lie outside of $H_{\ell, N}^d$ when setting $\k_i$ to be $- \ell_i -1$ are trivial or undefined. 

The purpose of this paper is to answer these questions and to study the structures of the Hahn 
polynomials for hypergeometric distribution. More specially, we want to understand the formal process 
of setting $\k_i$ as negative integers $- \ell_i -1$ rigorously and, in particular, to determine if and how 
the formal process leads to structural relations, such as generating functions, recurrence relations and difference equations,
for the Hahn polynomials on the polyhedra. The question has interesting implications even for the Hahn 
polynomials of one variable, which we briefly describe to facilitate the discussion. 

The classical Hahn polynomials $Q_m(x; a,b, N)$ are orthogonal with respect to a discrete inner product 
defined on the set $V_N=\{x \in \NN_0: 0 \le x \le N\}$ and $\{Q_m(x;a,b,N): 0\le m \le N\}$ is an orthogonal 
basis. Assume, for example, $\ell_1,\ell_2 \le N$ and $\ell_1+\ell_2 \ge N$. Then, for $d=1$, a basis of 
the orthogonal polynomials for the hypergeometric distribution is given by 
$\{Q_m(x; -\ell_1-1, - \ell_2-1, N):  0 \le m \le \ell_1\}$ and the orthogonality is defined on the set 
$V_{\ell,N}=\{x \in \NN: N - \ell_2 \le x \le \ell_1\}$. We shall show that all polynomials 
$Q_m(x; -\ell_1-1, -\ell_2-1, N)$ with $m >\ell_1$ contain the factor 
$\prod_{j=N-\ell_2}^{\ell_1} (x-j)$. In particular, this shows that all polynomials with index outside 
$H_{\ell,N} = \{0,1,\ldots, \ell_1\}$ vanishes on the set $V_{\ell,N}$. 
 
The situation for $d > 1$, however, is much more complicated. When we set $\k_i =  - \ell_i -1$, 
those classical Hahn polynomials with indices outside $H_{\ell,N}^d$ still vanish on the polyhedral
lattice set $V_{\ell,N}^d$; however, they do not always contain linear factors that vanish trivially on 
$V_{\ell,N}^d$. In fact, they lead to non-trivial, likely irreducible, polynomials that vanish on large subsets 
of lattice points. The complication for $d  > 1$ requires a careful consideration when deriving structural 
properties for the Hahn polynomials on polyhedra from those of classical Hahn polynomials. 

The paper is organized as follows. In the next section we consider the case $d =1$ and study Hahn 
polynomials with negative integer parameters.  The definition of the Hahn polynomials on the 
polyhedra is discussed in Section 3, which ends with examples of nontrivial Hahn polynomials 
that vanish on a large set of lattice points. The examples lead to the study of factorization of the Hahn 
polynomials of two variables in Section 4. The generating functions of the Hahn polynomials on the 
polyhedra are discussed in Section 5, whereas their bispectral properties are described in Section 6, 
which contains, in particular, explicit recurrence relations and difference equations satisfied by these polynomials.  

\section{Hahn polynomials with negative integer parameters}
\setcounter{equation}{0}

We start with a short review of the classical Hahn polynomials that depend on two real parameters 
$a, b > -1$. The main goal of this section is to study Hahn polynomials when $a$ and $b$ become 
negative integers, for which some of the properties of the Hahn polynomials remain valid whereas 
others become more subtle. 

\subsection{Classical Hahn polynomials}
For $a, b > -1$, the classical Hahn polynomials are ${}_3F_2$ hypergeometric functions given by
\begin{equation}\label{eq:HahnQ}
  Q_n(x; a, b, N) := {}_3 F_2 \left( \begin{matrix} -n, n+a + b+1, -x\\
       a+1, -N \end{matrix}; 1\right), \qquad n = 0, 1, \ldots, N.    
\end{equation}
They are discrete orthogonal polynomials with respect to the weight function
$$
   w_{a,b}(x;N) = \frac{N!}{(a+b+2)_N} \frac{(a+1)_x (b+1)_{N-x}}{x! (N-x)!}, \qquad x =0,1,\ldots, N, 
$$
over the set of integers $\{0,1,\ldots, N\}$, More precisely, they satisfy
\begin{align} \label{eq:orthoHahnQ}
 & \sum_{x=0}^N  Q_n(x; a, b, N) Q_m(x; a, b, N) w_{a,b}(x;N)\\
 &\quad = \frac{(-1)^n n! (b+1)_n (a+b+N+2)_{n} (n+a+b+1)}
     { (a+1)_n (a+b+2)_n (-N)_n (2n+a+b+1)} \delta_{n,m}, \quad n,m\le N. \notag
\end{align}
Furthermore, these polynomials can also be defined via a generating function   
\begin{equation} \label{eq:generatingHahn}
  (1+t)^N \frac{P_n^{(a,b)}(\frac{1-t}{1+t})}{P_n^{(a,b)}(1)} 
     = \sum_{x =0}^N \binom{N}{x} Q_n(x; a, b, N) t^x,
\end{equation}
where $P_n^{(a,b)}$ is the Jacobi polynomial defined by 
\begin{equation*}
    P_n^{(a,b)}(t)  =  \frac{(a+1)_n}{n!} {}_2F_1\left( \begin{matrix} -n, n+a + b+1 \\
       a+1 \end{matrix};  \frac{1-t}{2}\right).
\end{equation*}
They also satisfy the relation
\begin{equation}\label{eq:HahnQrev}
   Q_n(x; a, b, N) = (-1)^n\frac{(b+1)_n}{(a+1)_n} Q_n(N-x,b,a,N).       
\end{equation}

\subsection{Hahn polynomials with negative integer parameters}
Let $N$ be a positive integer. Let $\ell_1$ and $\ell_2$ be two positive integers that satisfy 
\begin{equation}\label{eq:ell-N}
  \ell_1+\ell_2 \ge N.
\end{equation}
We consider the hypergeometric distribution or the weight function
\begin{equation} \label{eq:H-weight}
  \sH_{\ell, N}(x) =   \frac{\binom{\ell_1}{x}\binom{\ell_2}{N-x}}{\binom{\ell_1+\ell_2}N}
     = \frac{N!}{(-\ell_1-\ell_2)_N} \frac{(-\ell_1)_x (-\ell_2)_{N-x}}{x! (N-x)!}, \quad x \in \NN_0.
\end{equation}
Throughout this paper, we shall adopt the following notation,
$$
    a \wedge b = \min\{a, b\} \quad \hbox{and}\quad a \vee b = \max\{a,b\}, \qquad a, b \in \RR.
$$

\begin{lem}
Let $N,  \ell_1$ and $\ell_2$ be positive integers satisfying \eqref{eq:ell-N}. Then
\begin{equation} \label{eq:H-weight2}
   \sH_{\ell,N}(x) =   \frac{(\ell_1 \vee N)!(\ell_2\vee N)!}{N! (-\ell_1-\ell_2)_N} 
   \frac{( -\ell_1 \wedge N)_x (-\ell_2 \wedge N)_{N-x}}{(x-N + \ell_2 \vee N)! (\ell_1 \vee N -x)!}.
\end{equation}
In particular, the function $\sH_{\ell, N}$ is positive and supported on the set 
$$
    V_{\ell,N} := \big \{x \in \NN_0:  N - \ell_2 \wedge N \le x \le \ell_1 \wedge N \big \}. 
$$
\end{lem}

\begin{proof}
If $\ell_1,  \ell_2 \le N$, then \eqref{eq:H-weight2} and \eqref{eq:H-weight} coincide. In all other cases,
we can rewrite $\sH_{\ell,N}$ given in \eqref{eq:H-weight} to the formula in \eqref{eq:H-weight2} by
using the identity $(-m)_n = (-1)^n m!/(m-n)!$ for $m, n \in \NN_0$. 
The claim on the support of $\sH_{\ell,N}$ follows readily from $(-m)_n > 0$ if $n \le m$ and $(-m)_n = 0$ if $n > m$.
\end{proof}

We consider orthogonal polynomials with respect to the inner product
\begin{align} \label{eq:inner-sQ}
   \la f,g\ra_{\ell,N} =  \sum_{x=  N -  \ell_2\wedge N }^{\ell_1\wedge N} f(x) g(x) \sH_{\ell,N}(x), 
\end{align}
which satisfies $\la 1,1\ra_{\ell,N} =1$. Applying the Gram-Schmidt process, we can identify a family of 
orthogonal polynomials $\{\sQ_n: 0 \le n \le \deg_{\ell,N}\}$, where 
$$
   \deg_{\ell,N} := \ell_1\wedge N +\ell_2 \wedge N - N.
$$
Evidently, $\sH_{\ell, N}(x) =  w_{-\ell_1-1,-\ell_2-1}(x;N)$. We make the following definition. 

\begin{defn}
Let $\ell_1, \ell_2$ and $N$ be positive integers such that $\ell_1+\ell_2 \ge N$. 
For $0 \le n \le \ell_1 \wedge N$, define Hahn polynomials with negative integer parameters by  
\begin{align}\label{eq:HahnQN}
  \sQ_n(x; \ell_1,\ell_2, N) \, & = Q_n(x; -\ell_1-1,-\ell_2-1,N) \\
      & = {}_3 F_2 \left( \begin{matrix} -n, n-\ell_1-\ell_2 -1, -x\\
       -\ell_1, - N \end{matrix}; 1\right),   \notag
\end{align}
where $Q_n$ is the classical Hahn polynomial \eqref{eq:HahnQ}. 
\end{defn}

For $0 \le n \le \deg_{\ell, N}$, the orthogonality of the polynomials $\sQ_n(\cdot;\ell_1,\ell_2,N)$ 
follows from that of classical Hahn polynomials. Indeed, the identity 
\eqref{eq:orthoHahnQ} involves only rational functions in $a, b$, so that we can apply analytic continuation 
to extend it to $a, b$ being negative integers, while the support set of $\sH_{\ell,N}$ shows that the inner 
product becomes \eqref{eq:inner-sQ}.

\begin{thm}
For $0 \le m, n \le \deg_{\ell, N}$, 
$$
 \big \langle \sQ_n(\cdot; \ell_1,\ell_2,N), \sQ_m(\cdot; \ell_1,\ell_2,N) \big \rangle_{\ell,N} = 
     \delta_{m,n}  \sB_n(\ell, N),
$$
where 
$$
  \sB_n(\ell,N) = \frac{(-1)^{n} n! (-\ell_2)_n(-\ell_1-\ell_2+N)_n (-n+\ell_1+\ell_2+1)}
  {(-\ell_1)_n(-\ell_1-\ell_2)_n (-N)_n (-2n+\ell_1+\ell_2+1)}.
$$
\end{thm}

From \eqref{eq:HahnQN}, the Hahn polynomials with negative integer parameters are well defined 
if $0 \le n \le  \ell_1\wedge N$. If $\ell_2 \ge N$, then $\deg_{\ell, N} = \ell_1\wedge N$. If $\ell_2 < N$, 
however, $\ell_1\wedge N > \deg_{\ell,N}$, we have more polynomials than needed. It turns out that the
extra polynomials are entirely zero when restricted on $V_{\ell,N}$. 
 
\begin{thm} \label{thm:factor1D}
 Assume $\ell_2 \le N$.  Let 
$n =\deg_{\ell,N} + m +1$ for $m  = 0,1,\ldots, N- \ell_2 -1$. Then
\begin{align} \label{eq:Qfactor1}
 \sQ_n(x; \ell_1,\ell_2, N) = \, &  \frac{ (-N+\ell_2+1)_m}{(-\ell_1\wedge N)_n}  \prod_{j=N-\ell_2}^{\ell_1\wedge N} (x-j)\\
     & \qquad\qquad \times \sQ_m(x, N-\ell_2-1,| N-\ell_1| -1, \ell_1 \vee N).  \notag
\end{align}
In particular, $\sQ_n(x;\ell_1,\ell_2,N)$ vanishes on $V_{\ell,N}$ if $\deg_{\ell,N} < n \le \ell_1\wedge N$. 
\end{thm}

\begin{proof}
The assumption on $m$ implies $n \le  \ell_1 \wedge N$, which implies that the constant in front of $\sQ_m$ 
in the righthand side of \eqref{eq:Qfactor1} is nonzero. 

First assume $\ell_1 \le N$. We need the following identity in \cite[Entry (7.4.4.83)]{PBM},
$$
{}_3 F_2 \left( \begin{matrix} -n, a,b\\ c, d \end{matrix}; 1\right)
    = \frac{(c+d-a-b)_n}{(c)_n}{}_3 F_2 \left( \begin{matrix} -n, d-a,d-b\\ d, c+d-a-b \end{matrix}; 1\right),
$$
where $n$ is a positive integer and the identity holds when both sides are finite. Choose 
$a=n-\ell_1-\ell_2 -1$, $b=-x$, $c=-\ell_1$ and $d= -N$. Then for our choice of $n =  \ell_1+\ell_2 - N + m +1$
we obtain
$$
 \sQ_n(x; \ell_1,\ell_2, N) = \frac{(-\ell_1-m+x)_n}{(-\ell_1)_n}
        {}_3 F_2 \left( \begin{matrix} -n, -m, -N+x \\ -\ell_1-m+x, -N \end{matrix}; 1\right).
$$
We write $(-\ell_1-m+x)_{\ell_1+\ell_2 - N + m +1} = (-\ell_1-m+x)_m (-\ell_1+x)_{\ell_1+\ell_2 - N +1}$. While
the second term is $(-\ell_1+x)_{\ell_1+\ell_2-N+1} =  \prod_{j=N-\ell_2}^{\ell_1} (x-j)$, the first term
combining with the $ {}_3 F_2 $ function gives, using the identity \cite[Entry (7.4.4.86)]{PBM}
$$
\frac{(c)_m}{(c-a)_m} {}_3 F_2 \left( \begin{matrix} -m, a,b\\ c, d \end{matrix}; 1\right)
     =   {}_3 F_2 \left( \begin{matrix} -m, a, d-b\\ d, a-c-m+1 \end{matrix}; 1\right)
$$
with $a = - N+x$, $b=-n$, $c= -\ell_1 - m +x$, and $d = -N$, that 
$$
\frac{(-\ell_1-m+x)_m}{(N-\ell_1-m)_m}
 {}_3 F_2 \left( \begin{matrix} -n, -m, -N+x \\ -\ell_1-m+x, -N \end{matrix}; 1\right)
 =   {}_3 F_2 \left( \begin{matrix}  -m, -N+n, -N+x \\ -N+\ell_1+1, -N \end{matrix}; 1\right).
$$
This last function can be identified with $\sQ_m(N-x; N-\ell_1-1, N-\ell_2-1, N)$, which we further write, using
the identity \eqref{eq:HahnQrev}, as
$$
 (-1)^m \frac{(-N+\ell_2+1)_m}{(-N+\ell_1+1)_m} \sQ_m(x; N-\ell_2-1, N-\ell_1-1, N).
$$
Putting these together proves the identity \eqref{eq:Qfactor1} when $\ell_1 \le N$. 

Next we assume $\ell_1 > N$. Exchanging the role of $\ell_1$ and $N$ in the ${}_3F_2$ of \eqref{eq:HahnQN},
we see that the Hahn polynomials satisfy
\begin{equation} \label{eq:sQHahn2}
      \sQ_n(x;\ell_1,\ell_2,N) =  \sQ_n(x; N,\ell_1+\ell_2-N, \ell_1).
\end{equation}
With $\wh \ell_1 = N$, $\wh \ell_2 =  \ell_1+\ell_2 - N$ and $\wh N = \ell_1$, the polynomial in the righthand
side is $\sQ_n(x; \wh \ell_1, \wh \ell_2, \wh N)$ with $\wh \ell_1 \le \wh N$ and $\wh \ell_1 + \wh \ell_2 \ge \wh N$,
which is factorable by what we proved in the previous paragraph. Hence, \eqref{eq:Qfactor1} for $\ell_1 > N$ 
follows from that for $\ell_1 \le N$. The proof is complete. 
\end{proof}

By symmetry, one may expect a factorization of $\sQ_n$ when $\ell_2 >  N \ge \ell_1$. Indeed, this
holds for the polynomial $\sQ_n(N-x; \ell_2,\ell_1,N)$, which is well defined for $n \le \ell_2$ and the factorization
holds for $n > \ell_1+\ell_2 -N$. However, by \eqref{eq:HahnQrev}, 
$$
    \sQ_n(x;\ell_1,\ell_2,N)  = (-1)^n \frac{(-\ell_2)_n}{(-\ell_1)_n} \sQ_n(N-x; \ell_2,\ell_1,N).
$$
The constant in the right-hand side makes sense only if $n \le \ell_1 < \ell_2$. Thus, it may seem that applying 
first \eqref{eq:HahnQrev} and then Theorem \ref{thm:factor1D} can lead to interesting new factorizations, 
but Theorem Theorem \ref{thm:factor1D} cannot be applied.

The Hahn polynomial $\sQ_m(x, N-\ell_2-1,| N-\ell_1| -1, \ell_1\vee N)$ in the righthand side of 
\eqref{eq:Qfactor1} still has negative integer parameters. Experiment with small $\ell_i$ and $N$ shows 
that many such polynomials are irreducible. One may ask if these polynomials are irreducible or 
irreducible after further factoring out the linear terms. The answer to both questions, however, is negative. 

\begin{exam}
If $(\ell_1,\ell_2,N) = (6,8,12)$, then \eqref{eq:Qfactor1} becomes 
$$
   \sQ_6(x;6,8,12) = - \frac{1}{5!}(x-4)(x-5)(x-6)\sQ_3(x;3,5,12) 
$$
and the Hahn polynomial in the righthand side contains a further linear factor, 
$$
 \sQ_3(x;3,5,12)=- \frac{1}{132}(x-4) (2 x^2- 13 x +33). 
$$ 
Furthermore, if $(\ell_1,\ell_2,N) = (8,9,16)$, then \eqref{eq:Qfactor1} becomes 
$$
   \sQ_7(x;8,9,16) = \frac{1}{56}(x-7)(x-8)\sQ_5(x;6,7,16) 
$$
and the Hahn polynomial in the righthand side is not irreducible, 
$$
\sQ_5(x;6,7,16) =  -\frac{1}{24960} (52 - 14 x + x^2) (-480 + 159 x - 20 x^2 + x^3).
$$
\end{exam}

\subsection{Generating function} 
The generating function \eqref{eq:generatingHahn} can also be adopted for the Hahn polynomials
of negative integer parameters. When $a$ and $b$ are negative integers, the Jacobi polynomial 
$P_n^{(a,b)}$ are known to have a degree reduction for some $n$. Indeed, if $n$ is a positive 
integer and $n+ m+ a+b =0$, $m$ is an integer, $1 \le m \le n$, then \cite[(4.22.3)]{Sz}
$$
   \binom{n}{m-1} P_n^{(a,b)}(t) = \binom{n+a}{n-m+1} P_{m-1}^{(a,b)}(t). 
$$
However, we claim that this degree reduction will be irrelevant in our setting. Indeed, for $\ell_1$ and $\ell_2$ 
satisfy \eqref{eq:ell-N}, we are interested in $n$ that satisfies $0 \le n \le \deg_{\ell,N}$. With $a = -\ell_1-1$, 
$b = -\ell_2-1$, we have $m= -n-a-b = \ell_1 +\ell_2 + 2 -n \ge  \ell_1 +\ell_2 + 2 -\deg_{\ell,N}$. If 
$N \le \ell_2$, then $\deg_{\ell,N} = \ell_1\wedge N$, so tha $m \ge  \ell_1+N+2-\ell_1\wedge N \ge N+2 > n$,
whereas if $\ell_2 \le N$, then $\deg_{\ell,N} = \ell_1\wedge N - (N-\ell_2)$, so that 
$m \ge \ell_1+N+2-\ell_1\wedge N \ge N+2 >n$. This verifies the claim. To simplify notation, 
we make the following definition.

\begin{defn}
Let $\ell_1,\ell_2$ be positive integers. For $n =0,1,\ldots,  \ell_1 \wedge \ell_2$, we define 
$$
  G_n^{(\ell_1,\ell_2)}(t) = \frac{P_n^{(-\ell_1-1,-\ell_2-1)}(t)}{P_n^{(-\ell_1-1,-\ell_2-1)}(1)} = 
      {}_2F_1\left( \begin{matrix} -n, n-\ell_1-\ell_2-1 \\
          -\ell_1 \end{matrix};  \frac{1-t}{2}\right). 
$$
\end{defn}

Since $P_n^{(-\ell_1-1,-\ell_2-1)}(1)= \frac{(-\ell_1)_n}{n!}$ and $P_n^{(a,b)}(-t)= (-1)^n P_n^{(b,a)}(t)$, 
we have
\begin{equation}\label{eq:Gn(-t)}
G_n^{(\ell_1,\ell_2)}(t)  = (-1)^n \frac{(-\ell_2)_n}{(-\ell_1)_n} G_n^{(\ell_2,\ell_1)}(-t).
\end{equation}

By analytic continuation, the identity \eqref{eq:generatingHahn} remains valid when 
$a= - \ell_1 -1$ and $b= -\ell_2-1$, which gives 
\begin{equation} \label{eq:genHahnQ1}
  (1+t)^N  G_n^{(\ell_1,\ell_2)}\left(\frac{1-t}{1+t}\right) 
             = \sum_{x =0}^N \binom{N}{x} \sQ_n(x; \ell_1, \ell_2, N) t^x. 
\end{equation}
This works for all $\ell_1 + \ell_2 \ge N$. Its right-hand side, however, sums over all integers in $[0,N]$
instead of over integers in $[N- \ell_2 \wedge N, \ell_1\wedge N]$, on which the orthogonality of 
$\sQ_n(x; \ell_1, \ell_2, N)$ is defined. A more general generating function is given below. 

\begin{prop} \label{prop:GenF1d}
Let $\a = \ell_1+\ell_2\wedge N - \ell_1\wedge N$ and $\b =  \ell_2+\ell_1\wedge N - \ell_2\wedge N$. Then
for $0 \le n \le \deg_{\ell,N}$, 
 \begin{align} \label{eq:generatingHahnQ}
  b_{\ell,N}t^{N-\ell_2\wedge N} & (1+t)^{\ell_1\wedge N+\ell_2\wedge N - N}  G_n^{(\a,\b)} \left(\frac{1-t}{1+t}\right) \\
            & = \sum_{x =N-\ell_2 \wedge N}^{\ell_1\wedge N} \binom{\ell_1\wedge N+\ell_2\wedge N - N}{\ell_1 \wedge N - x} \sQ_n(x; \ell_1, \ell_2, N) t^x, \notag
\end{align}
where $b_{\ell,N}$ is the constant given by 
$$
  b_{\ell,N} = \frac{(-\ell_2)_n (-\ell_1-\ell_2\wedge N + N)_n}{(-\ell_1)_n (-\ell_2+\ell_2\wedge N - N)_n}.
$$ 
In particular, $b_{\ell,N} =1$ if $\ell_2 \ge N$. 
\end{prop}

\begin{proof}
The identity \eqref{eq:genHahnQ1} gives \eqref{eq:generatingHahnQ} when $\ell_1 \ge N$ and $\ell_2 \ge N$. 
In all other cases, we need to modify the righthand side so that the summation is over the integers in the 
interval $[N-\ell_2 \wedge N, \ell_1\wedge N]$. Assume $\ell_2 \ge N$. 
The identity \eqref{eq:genHahnQ1} for $\sQ_n$ in the right-hand side of \eqref{eq:sQHahn2} gives 
\eqref{eq:generatingHahnQ} for $\ell_1< N$. 

Next we assume $\ell_2 \le N$. By \eqref{eq:HahnQrev}, 
\begin{equation} \label{eq:sQHahn3}
  \sQ_n(x,\ell_1,\ell_2,N) =  (-1)^n \frac{(-\ell_2)_n}{(-\ell_1)_n} \sQ_n(N-x, \ell_2, \ell_1, N).
\end{equation}
Applying \eqref{eq:sQHahn2} on the Hahn polynomial in the right hand side, we deduce
$$
 \sQ_n(x,\ell_1,\ell_2,N) = (-1)^n \frac{(-\ell_2)_n}{(-\ell_1)_n} \sQ_n(N-x, N, \ell_1+\ell_2 - N, \ell_2).
$$ 
Using this identity, the right-hand side of \eqref{eq:generatingHahnQ} for $\ell_1 \ge N$ becomes
\begin{align*}
 \sum_{x=N-\ell_2}^N & \binom{\ell_2}{N-x}\sQ_n(x;\ell_1,\ell_2,N)t^x =
     \sum_{x= 0}^{\ell_2} \binom{\ell_2}{x}\sQ_n(N-x;\ell_1,\ell_2,N)t^{N-x}\\
  & = (-1)^n \frac{(-\ell_2)_n}{(-\ell_1)_n}\sum_{x= 0}^{\ell_2} \binom{\ell_2}{x}\sQ_n(x;N,\ell_1+\ell_2-N,\ell_2)t^{N-x} \\
  & = (-1)^n \frac{(-\ell_2)_n}{(-\ell_1)_n} t^N (1+\tfrac1t)^{\ell_2} 
        G_n^{(N,\ell_1+\ell_2-N)}\left(\frac{1- \frac1t}{1+\frac1t}\right) \\
  & = (-1)^n \frac{(-\ell_2)_n}{(-\ell_1)_n} t^{N-\ell_2} (1+t)^{\ell_2} 
           G_n^{(N, \ell_1+\ell_2-N)}\left(- \frac{1-t}{1+t}\right),   
\end{align*}
where the third equality follows from \eqref{eq:genHahnQ1}. Finally, applying \eqref{eq:Gn(-t)}, we have
established \eqref{eq:generatingHahnQ} for $\ell_1 \ge N$ and $\ell_2 \le N$. In the last case, when
$\ell_1 \le N$ and $\ell_2 \le N$, we rewrite the Hahn polynomial, starting from \eqref{eq:sQHahn2},
using \eqref{eq:sQHahn3} and then \eqref{eq:sQHahn2} one more time, to obtain
\begin{align*}
 \sQ_n(x;\ell_1,\ell_2,N)& = (-1)^n \frac{(N-\ell_1-\ell_2)_n}{(-N)_n} \sQ_n(\ell_1-x; \ell_1+\ell_2 - N, N,\ell_1)\\
         & = (-1)^n \frac{(N-\ell_1-\ell_2)_n}{(-N)_n} \sQ_n(\ell_1-x; \ell_1,\ell_2, \ell_1+\ell_2-N).
\end{align*} 
With this identity, the righthand side of \eqref{eq:generatingHahnQ} for $\ell_1 \le N$ becomes
\begin{align*}
& \sum_{x=N-\ell_2}^{\ell_1}  \binom{\ell_1+\ell_2-N}{\ell_1-x}\sQ_n(x;\ell_1,\ell_2,N)t^x \\
  & \quad=
     \sum_{x= 0}^{\ell_1+\ell_2-N} \binom{\ell_1+\ell_2-N}{x} \sQ_n(\ell_1-x;\ell_1,\ell_2,N)t^{\ell_1-x}\\
  & \quad= (-1)^n \frac{(N-\ell_1-\ell_2)_n}{(-N)_n} \sum_{x= 0}^{\ell_1+\ell_2-N} \binom{\ell_1+\ell_2-N}{x}
         \sQ_n(x; \ell_1,\ell_2, \ell_1+\ell_2-N)t^{\ell_1-x},
\end{align*}
from which the proof can be completed as in the case $\ell_1 \ge N$ and $\ell_2 \le N$. 
This completes the proof.
\end{proof}

The usual orthogonality of the Jacobi polynomials, however, no longer holds when the parameters are 
negative integers, since $(1-x)^a(1+x)^b$ is not integrable on $[-1,1]$ if $a$ and/or $b$ are negative 
integers. 
It is possible, however, to define a linear functional $\CL$, so that the polynomials $G_n^{(\ell_1,\ell_2)}$ 
are orthogonal in the sense that $\CL\big( G_m^{(\ell_1,\ell_2)} G_n^{(\ell_1,\ell_2)}\big) =0$ for $m \ne n$,
although this linear function is no longer positive definite. 

Let $\ell_1$ and $\ell_2$ be positive integers. We define the linear functional $\CL$
on the space $\Pi_{\ell_1+\ell_2}$ of polynomials of degree at most $\ell_1+\ell_2$ 
so that its moments are given by
$$
  \CL(x^k) = {}_2F_1\left( \begin{matrix} - k, -\ell_1\\ -\ell_1-\ell_2 \end{matrix}; 2\right), \qquad 0 \le k \le \ell_1+\ell_2.
$$

\begin{prop}
Let $\ell_1$ and $\ell_2$ be two positive integer. Then the polynomials $G_k^{(\ell_1, \ell_2)}$ satisfy the orthogonal relation 
$$
  \CL\left(G_m^{(\ell_1, \ell_2)} G_n^{(\ell_1, \ell_2)}\right) =  h_n^{(\ell_1,\ell_2)}\delta_{m,n}, 
  \qquad 0\le m, n \le \ell_1 \wedge \ell_2, 
$$
where the constants $h_n^{(\ell_1,\ell_2)}$ are given by
$$
h_n^{(\ell_1,\ell_2)} = \frac{n! (-\ell_2)_n (1+\ell_1+\ell_2 -n)}{(-\ell_1)_n (-\ell_1-\ell_2)_n (1+\ell_1+\ell_2 - 2n)}.
$$
\end{prop}

\begin{proof}
First we claim that the following relation holds
\begin{equation}\label{eq:moments}
 \CL \left( \big( \tfrac{1-x}{2}\big)^m\right )  = \frac{(-\ell_1)_m}{(- \ell_1-\ell_2)_m}, \qquad m =0,1, \ldots, \ell_1+\ell_2.
\end{equation}
Indeed, by the binomial formula and the definition of the ${}_2F_1$, we have 
\begin{align*}
\CL \left( \big( \tfrac{1-x}{2}\big)^m\right ) \, &=  \frac{1}{2^m} \sum_{k=0}^m \binom{m}{k}(-1)^k \CL(x^k) \\
      & = \frac{1}{2^m} \sum_{k=0}^m \frac{(-m)_k}{k!} \sum_{j=0}^k  \frac{(-k)_j(-\ell_1)_j} { j! (-\ell_1-\ell_2)_j} 2^j \\
      & = \frac{1}{2^m} \sum_{j=0}^m  \frac{(-\ell_1)_j 2^j }{ j! (-\ell_1-\ell_2)_j} \sum_{k=j}^m  \frac{(-m)_k (-k)_j}{k!}.
\end{align*}
Changing summation index and rewriting, it is easy to see that 
$$
\sum_{k=j}^m  \frac{(-m)_k (-k)_j}{k!} = (-1)^j(-m)_j \sum_{k=0}^{m-j} \frac{(-m+j)_k}{k!} =  m!\delta_{j,m}, 
$$
from which the claimed formula \eqref{eq:moments} follows immediately. 

Now let $n \le \ell_1\wedge \ell_2$. For $0 \le m \le n$, we use \eqref{eq:moments} and the identity
$(a)_{m+k} = (a)_m (a+m)_k$ to obtain 
\begin{align*}
  \CL \big(G_n^{(\ell_1,\ell_2)}(x) \big(\tfrac{1-x}{2}\big)^m\big) & = 
   \sum_{k=0}^n  \frac{(-n)_k (n-\ell_1-\ell_2-1)_k}{k! (-\ell_1)_k} \CL\big((\tfrac{1-x}{2})^{m+k}\big) \\
      & = \frac{(-\ell_1)_m}{(-\ell_1-\ell_2)_m}
     \sum_{k=0}^n  \frac{(-n)_k (n-\ell_1-\ell_2-1)_k(-\ell_1+m)_k}{k! (-\ell_1)_k(- \ell_1-\ell_2+m)_k}  \\
      & = \frac{(-\ell_1)_m}{(-\ell_1-\ell_2)_m} 
       {}_3F_2\left( \begin{matrix} - n, n-\ell_1 - \ell_2 -1, -\ell_1+m \\ -\ell_1, -\ell_1-\ell_2+m \end{matrix}; 1\right).
\end{align*}
This hypergeometric function is a balanced terminating ${}_3F_2$ and, by Saalsch\"utz summation formula, 
we conclude that 
\begin{align*}
  \CL \big(G_n^{(\ell_1,\ell_2)}(x) \big(\tfrac{1-x}{2}\big)^m\big) \, & = \frac{(-\ell_1)_m}{(-\ell_1-\ell_2)_m} 
      \frac{(1+\ell_2-n)_n (-m)_n}{ (-\ell_1)_n (1+\ell_1+\ell_2-n-m)_n} \\
   & = \frac{(-1)^n n! (-\ell_2)_n}{(-\ell_1-\ell_2)_{2n}}\delta_{m,n},
\end{align*}
where we have used $(-m)_n =0$ for $m < n$ and $(-n)_n =(-1)^n n!$. This proves that $G_n^{(\ell_1,\ell_2)}$ is orthogonal to $(1-x)^m$ for $m < n$ and hence, by linearity, to $G_m^{(\ell_1,\ell_2)}$.  Furthermore, multiplying the above identity by 
the coefficient $(-1)^n(n-\ell_1-\ell_2-1)_n /(-\ell_1)_n$ of $ \big(\tfrac{1-x}{2}\big)^n$ in $G_n^{(\ell_1,\ell_2)}$ verifies the 
formula for $h_n^{(\ell_1,\ell_2)}$. 
\end{proof}

Since $\sh_n^{(\ell_1,\ell_2)}$ has the sign $(-1)^n$, the moment functional $\CL$ is not positive definite. 
This can also be seen in the coefficients of the three-term relations satisfied by $t G_n^{(\ell_1,\ell_2)}(t)$,
where the coefficients of $ G_{n+1}^{(\ell_1,\ell_2)}$ and $G_{n-1}^{(\ell_1,\ell_2)}$ have opposite signs. 

The generating function \eqref{eq:generatingHahnQ} of the Hahn polynomials of negative parameters requires
$0\le n \le \deg_{\ell,N} = \ell_1\wedge N +\ell_2 \wedge N -N$. Since $\deg_{\ell,N} \le \ell_1\wedge \ell_2$, we see that all $G_n^{(\ell_1,\ell_2)}$ in the 
generating function \eqref{eq:generatingHahnQ} are orthogonal polynomials. 

Finally, let us mention that the three-term relation satisfied by $\sQ_n(\cdot,\ell_1,\ell_2,N)$ can be deduced
from that of the classical Hahn polynomials, which shows 
\begin{equation} \label{eq:3-term}
  x \phi_n  = - A_n \phi_{n+1} + (A_n+C_n) \phi_n  -  C_n \phi_{n-1}, \quad 0 \le n \le \deg_{\ell,N}-1,
\end{equation}
where $\phi_n = \sQ_n(\cdot,\ell_1,\ell_2,N)$ and the coefficients $A_n$ and $C_n$ are given by
\begin{align*}
  A_n \, & = \frac{(n -\ell_1-\ell_2-1) (n -\ell_1) (N - n)}{(2 n -\ell_1-\ell_2- 1) (2 n -\ell_1-\ell_2)}, \\
  C_n\, & = \frac{n (n + N -\ell_1-\ell_2-1) (n -\ell_2-1)}{(2 n -\ell_1-\ell_2- 1) (2 n -\ell_1-\ell_2-2)}.
\end{align*}
In our statement of \eqref{eq:3-term}, we assume $n \le \deg_{\ell,N} -1$. For $\ell_2 \le N$, the identity
also holds for $n=\deg_{\ell,N}$ by Theorem \ref{thm:factor1D} whenever $A_n$ and $C_n$ are finite. 
Since $\deg_{\ell,N} \le \frac{\ell_1+\ell_2}{2}$, the coefficients are well defined unless $n = \deg_{\ell,N}
=\frac{\ell_1+\ell_2}2$. This last equation is attained if $\ell_1=\ell_2=N$ and $n = \deg_{\ell,N} = N$, which 
leads to a pole in $A_n$ so that \eqref{eq:3-term} fails for $n = \deg_{\ell,N}$ in this particular case. 

\section{Hahn polynomials for hypergeometric distribution}
\setcounter{equation}{0}

Classical Hahn polynomials in several variables are those on lattice points inside a simplex. A brief 
review of these polynomials will be given in the first subsection. When their parameters become negative
integers, these polynomials become orthogonal polynomials for hypergeometric distribution, which will
be discussed in the second subsection. 

\subsection{Classical Hahn polynomials of several variables}
Let $N$ be a positive integer. Recall that $V_N^d$ is the set of lattice points in a discrete simplex
\begin{equation}\label{eq:VNd}
   V_N^d: = \{\nu \in \NN_0^d: |\nu| \le N\},
\end{equation}
and, for $\k \in \RR^{d+1}$ with $\k_i > -1$, $1 \le i \le d+1$, the function $W_{\k,N}$ defined in
\eqref{eq:HK-weight} is the normalized Hahn weigh function.  
The Hahn polynomials are orthogonal with respect to the discrete inner product  
\begin{equation*}
     \la f, g\ra_{W_{\k,N}} = \sum_{x \in V_N^d} f(x) g(x)W_{\k,N}(x).
\end{equation*}
For $0\le n \le N$, let $\CV_n^d(W_{\k,N})$ denote the space of orthogonal polynomials of degree $n$ 
with respect to this inner product. Then 
$$
   \dim \CV_n^d\left(W_{\k,N}\right) = \binom{n+d-1}{n}, \qquad n =0,1,2,\ldots. 
$$
An orthogonal basis of $\CV_n^d(W_{\k,N})$ can be given in terms of 
classical Hahn polynomials of one variable, for which we need the following notation:
 
For $y=(y_1,\ldots, y_{d}) 
\in \RR^{d}$ and $1 \le j \le d$, we define 
\begin{equation}\label{xsupj}
    \yb_j := (y_1, \ldots, y_j) \quad \hbox{and}\quad \yb^j := (y_j, \ldots, y_d), 
\end{equation}
and also define $\yb_0 := 0$ and $\yb^{d+1} := 0$. It follows that $\yb_d = \yb^1 = y$, 
and 
$$
   |\yb_j| = y_1 + \cdots + y_j,   \quad |\yb^j| = y_j + \cdots + y_d, \quad\hbox{and}\quad
   |\yb_0| = |\yb^{d+1}| = 0.
$$
For $\k = (\k_1,\ldots, \k_{d+1})$, we have $\bka^j := (\k_j, \ldots, \k_{d+1})$ for $1 \le j \le d+1$. Furthermore,
for $\nu \in \NN_0^d$ and $\k \in \RR^{d+1}$, we define
\begin{equation}\label{eq:aj}
   a_j:=a_j(\kappa,\nu):=|\bka^{j+1}| + 2 |\nu^{j+1}| + d-j, \qquad 1 \le j \le d.
\end{equation} 
Notice that $a_{d} = \k_{d+1}$ since $|\nu^{d+1}| =0$ by definition. 

\begin{prop} \label{prop:Hahn1}
For $x \in \ZZ_N^{d+1}$ and $\nu \in \NN_0^d$, $|\nu| \le N$, define
\begin{align}\label{eq:Hn-pos}
  Q_\nu(x;\kappa, N) =&  \prod_{j=1}^d  (-N+|\xb_{j-1}|+|\nu^{j+1}|)_{\nu_j}  \\ 
    &\qquad  \times  
    Q_{\nu_j}(x_j; \kappa_j, a_j, N- |\xb_{j-1}|-|\nu^{j+1}|). \notag
\end{align}
The polynomials in $\{Q_\nu(x; \kappa,N): |\nu| = n\}$ form a mutually orthogonal basis of $\CV_n^d(W_{\k,N})$ 
and $B_\nu := \la Q_\nu(\cdot; \kappa,N), \ Q_\nu(\cdot; \kappa,N) \ra_{W_{\k,N}}$ is given by, setting 
$\l_\k : = |\k|+d+1$, 
\begin{align*} 
 B_\nu(\k, N)  :=\frac{(-1)^{|\nu|}(-N)_{|\nu|}(\l_k)_{N+|\nu|}}
   { (\l_k)_N (\l_k)_{2|\nu|}} 
   \prod_{j=1}^d \frac{ (a_j+1)_{\nu_j}(\k_j+a_j+1)_{2\nu_j}  \nu_j! }
    { (\k_j+1)_{\nu_j} (\kappa_j+a_j + 1)_{\nu_j}}.
\end{align*}
\end{prop} 

These Hahn polynomials of several variables were defined and studied in \cite{KM} through 
a generating function that was later recognized as the Jacobi polynomial $P_\nu(x)$ on the 
simplex $T^d$ defined by
$$
T^d: = \{x \in \RR^d: x_1 \ge 0,\ldots, x_d \ge 0, |x|\le 1\}.
$$ 

\begin{prop}
For $\k \in \RR^{d+1}$ with $\k_i > -1$, $\nu \in \NN_0^d$ and $x \in \RR^d$, define
\begin{equation}\label{eq:Pnu}
P_\nu(x) =P_\nu^\kappa (x) := \prod_{j=1}^d \left(1-|\xb_{j-1}| \right)^{\nu_j}  \frac{P_{\nu_j}^{(a_j,\kappa_j)}\left 
  (\frac{2x_j}{1-|\xb_{j-1}|} -1\right)}{P_{\nu_j}^{(a_j,\kappa_j)}(1)}, 
\end{equation}
where $a_j = a_j(\k, \nu)$ is defined in \eqref{eq:aj}. Then the polynomials in $\{P_\nu:|\nu|=n\}$ form an 
orthogonal basis of $\CV_n^d(W_\k)$, where
\begin{equation}\label{eq:weightW}
  W_\kappa (x):=  x_1^{\kappa_1} \cdots x_d^{\kappa_d} (1-|x|)^{\kappa_{d+1}}.
\end{equation}
\end{prop}

These polynomials serve as generating functions of the Hahn polynomials, for which it is more convenient to
use a different normalization of the Hahn polynomials, denoted by $H_{\nu}(\cdot; \k, N)$, given by 
\begin{equation} \label{eq:HahnH}
 H_{\nu}(\a; \k, N) =  \frac{(-1)^{|\nu|}}{(-N)_{|\nu|}} \prod_{j=1}^d  \frac{(\kappa_j+1)_{\nu_j}}{(a_j+1)_{\nu_j}}  \, 
          Q_\nu(\a'; \kappa,N), 
\end{equation}
where we use homogeneous coordinates $\a = (\a',N-|\a|) \in \NN_0^{d+1}$. 

\begin{thm} \label{def:Hahn}
Let $\k\in\RR^{d+1}$ with $\kappa_i>-1$ and $N\in\NN$. For $\nu\in\NN_0^d$, $|\nu|\le N$,  
the Hahn polynomials $\sQ_\nu(\a; \k ,N)$ satisfy  
\begin{equation}\label{Hahngenfunc}
 P_{\nu,N}(y) = |y|^N P_\nu\Big ( \f {y'} {|y|} \Big)
   =      \sum_{|\alpha| = N} \frac{N!}{\a!} H_\nu(\alpha; \kappa,N)y^\alpha, 
\end{equation}
where $\a \in \NN_0^{d+1}$, $\a! = \a_1!\cdots \a_{d+1}!$ and $y = (y', y_{d+1}) \in \RR^{d+1}$. 
\end{thm}

The generating function serves as a starting point and an essential tool in the study of \cite{IX17, X15}. 

\subsection{Hahn polynomials with negative integer parameters}
Let $N \in \NN_0$ and $\ell_i \in \NN$ for $1 \le i \le d+1$. We assume that they satisfy 
\begin{equation} \label{eq:l-condition}
  \ell_i \le N \quad \hbox{and}\quad \ell_i+ \ell_{j} \ge N, \quad i \ne j, \quad 1 \le i, j \le d+1. 
\end{equation}
Recall that $V_{\ell,N}^d$, defined in \eqref{eq:poly-domain}, denotes the discrete polyhedral domain, which
we restate below,
\begin{equation*}
   V_{\ell, N}^d: = \left \{x \in \NN_0^d: 0 \le x_i \le \ell_i, \, 1 \le i \le d, \, \hbox{and}\, N-\ell_{d+1} \le |x| \le N\right\},
\end{equation*}
where $|x|=x_1+\cdots +x_d$. Evidently, $V_{\ell,N}^d$ is the simplex $V_N^d$ if all $\ell_i = N$. The polyhedral
is the simplex $V_N^d$ with its corners sliced off. In particular,
$$
  |V_{\ell,N}^d| =\binom{N+d}{d}-\sum_{k=1}^{d+1}\binom{N-\ell_k+d-1}{d},
$$
where we denote by $|E|$ the cardinality of the discrete set $E$. For $\ell$ and $N$ satisfying 
\eqref{eq:l-condition}, the hypergeometric distribution in $d$ variables is defined by 
\eqref{eq:poly-weight}, which is restated below, 
\begin{equation*}
  \sH_{\ell,N}(x)  = \frac{1}{\binom{|\ell |}N} \prod_{i=1}^d \binom{\ell_i}{x_i}\binom{\ell_{d+1}}{N-|x|}
 =  \frac{N!}{(-|\ell |)_N}\prod_{i=1}^d \frac{(-\ell_i)_{x_i}}{x_i!} \frac{(-\ell_{d+1})_{N-|x|}}{(N-|x|)!}. 
\end{equation*}
This defines a probability measure on $V_{\ell,N}^d$. The first identity is used more commonly in the 
probability theory; see, for example, \cite{JKB}. The Hahn polynomials for this distribution are orthogonal 
polynomials with respect to the inner product 
\begin{equation} \label{eq:Hahn-ipd}
  \la f, g\ra_{\ell,N} =  \sum_{x \in V_{\ell,N}^d} f(x) g(x)  \sH_{\ell,N} (x).
\end{equation}
Let $\Pi_N^d$ denote the space of polynomials of total degree at most $N$ in $d$ variables. Let $\CI(V_{\ell,N}^d)$ 
denote the ideal of polynomials that vanish on $V_{\ell,N}^d$. It is known that the space of orthogonal polynomials,
denoted by $\Pi_{\ell,N}^d$, with respect to $\la \cdot,\cdot\ra_{\ell,N}$ satisfies 
$$
\Pi_{\ell,N}^d =  \RR[x_1,\ldots,x_d]/ \CI(V_{\ell,N}^d).
$$

Since $\sH_{\ell,N}(x) = W_{-\ell -\one,N}(x)$ where $\one = (1,1,\ldots,1)$ and $\sH_{\ell,N}(x) =0$ if 
$x \in V_N^d \setminus V_{\ell,N}^d$, orthogonal polynomials with respect to the inner product 
$\la \cdot,\cdot\ra_{\ell,N}$ can be deduced from $Q_\nu(x; \k, N)$ in \eqref{eq:Hn-pos} by setting 
$\k = -\ell -\one$ for $\nu$ in an appropriate subset of $\{\nu: |\nu| \le N\}$. This narrative was carried out
in \cite{IX20}. Let us now rewrite the polynomials $Q_\nu(x; -\ell -\one, N)$ in terms of the Hahn polynomials 
with negative integer parameters. 

With the notation $\bell^j := (\ell_j, \ldots, \ell_{d+1})$ for $1 \le j \le d+1$, we define
\begin{equation}\label{eq:saj}
   \za_j:= - a_j(-\ell - \one,\nu) -1 = |\bell^{j+1}| - 2 |\bnu^{j+1}|, \qquad 1 \le j \le d.
\end{equation}

\begin{defn}
Let $\ell \in \NN_0^{d+1}$ and $N \in \NN$. Assume $\ell_i \le N$ and \eqref{eq:l-condition}. 
We define Hahn polynomials with negative integer parameters on the polyhedron $V_{\ell,N}^d$ by 
\begin{align*} 
 \sQ_\nu(x;\ell, N)  \, & = Q_\nu(x; -\ell -\one, N) \\
  & = \prod_{j=1}^d 
     (-N+|\xb_{j-1}|+|\bnu^{j+1}|)_{\nu_j} \sQ_{\nu_j}\left(x_j;  \ell_j, \za_j, N- |\xb_{j-1}|-|\bnu^{j+1}|\right).
\end{align*}
\end{defn}

In \cite{IX20}, we considered these polynomials for $\nu$ in the index set $H_{\ell,N}^d$ defined by 
\begin{align*}
  H_{\ell,N}^d := \left \{   \nu \in \NN_0^d:   |\nu| \le N, \, |\nu| \le |\ell|-N, 
    \nu_j \le \ell_j, \,   \nu_j \le \za_j, \, 1 \le j \le d \right \}.
\end{align*} 
The main result of \cite{IX20} states that they form an orthogonal basis of the space $\Pi_{\ell,N}^d$. 

\begin{thm} \label{thm:OP-polyhedra}
Let $N \in \NN$ and let $\ell_i \in \NN$ satisfy \eqref{eq:l-condition}. Then 
\begin{enumerate}[\quad \rm (i)]
\item The polynomials $\sQ_\nu(\cdot; \ell,N)$ are orthogonal and satisfy 
$$
   \left \langle   \sQ_\nu(\cdot; \ell,N),   \sQ_\mu(\cdot; \ell,N)  \right \rangle_{\ell,N} 
       =   \sB_\nu(\ell, N) \delta_{\nu,\mu}
$$
for all $\nu,\mu \in H_{\ell,N}^d$, where 
\begin{align} \label{eq:Bnu}
  \sB_\nu(\ell, N)  :=\frac{(-1)^{|\nu|}(-N)_{|\nu|} (-|\ell |)_{N+|\nu|}} { (-|\ell|)_N (-|\ell|)_{2|\nu|}} 
   \prod_{j=1}^d \frac{ (- \za_j)_{\nu_j} (-\ell_j-\za_j-1)_{2\nu_j}  \nu_j! }
       {  (-\ell_j)_{\nu_j} (-\ell_j-\za_j-1)_{\nu_j}}.
\end{align}
\item The set $\{  \sQ_\nu(\cdot ; \ell ,N): \nu \in H_{\ell,N}^d \}$ is a basis of $\Pi_{\ell,N}^d$. In particular,  
$$
 |H_{\ell,N}^d|=|V_{\ell,N}^d| = \dim \Pi_{\ell, N}^d.
$$
\end{enumerate}
\end{thm}

The part (i) follows from the classical Hahn polynomials on $V_N^d$ as we indicated above. The proof of
part (ii) is highly non-trivial and requires a rather involved combinatorial proof. The norm $\sB_\nu(\ell,N)$ is 
nonzero, in fact positive, for $\nu \in H_{\ell,N}^d$, which is how the set $H_{\ell,N}^d$ is conceived
and defined. 

The definition of the polynomials $\sQ_\nu(\cdot; \ell,N)$ shows that these polynomials are well defined
for $\nu$ in a set $\CH_{\ell,N}^d$ that contains $H_{\ell,N}^d$ as a subset. The next theorem is a complementary result of  
Theorem \ref{thm:OP-polyhedra}.

\begin{thm}\label{thm:CH_ellN}
Let $N \in \NN$ and $\ell_i \in \NN$ satisfying \eqref{eq:l-condition}. The polynomials $\sQ_\nu(\cdot; \ell,N)$
are well defined on $V_{\ell,N}^d$ if $\nu \in \CH_{\ell,N}^d$, where 
\begin{equation} \label{eq:CH_ellN}
      \CH_{\ell,N}^d =\{ \nu \in \NN_0^d:   \nu_i \le \ell_i, \, 1 \le i \le d,\,  |\nu| \le N\}. 
\end{equation}
Furthermore, if $\nu \in \CH_{\ell,N}^d \setminus H_{\ell,N}^d$, then the polynomials $\sQ_{\nu}(\cdot; \ell, N)$ 
vanishes on $V_{\ell,d}$. 
\end{thm}

\begin{proof}
Let $\nu \in  \CH_{\ell,N}^d$. For $0\le k \le \nu_j$, expanding $\sQ_{\nu_j}$ in $\sQ_\nu$ using its ${}_3F_2$ 
definition, and using the identity  
$$
   \frac{(-N+|\xb_{j-1}|+|\bnu^{j+1}|)_{\nu_j}}{(-N+|\xb_{j-1}|+|\bnu^{j+1}|)_k} = (-N+|\xb_{j-1}|+|\bnu^{j+1}|+k)_{\nu_j-k},  
$$
it is easy to see that $\sQ_{\nu}$ is well defined on $ \CH_{\ell,N}^d$. Since the orthogonal relation in 
Theorem \ref{thm:OP-polyhedra} is derived by setting $\k = -\ell -\one$, it holds for $\nu \in \CH_{\ell,N}^d$. 
In particular, the expression of $\sB_\nu(\ell,N)$ in \eqref{eq:Bnu} shows that the norm of 
$\sQ_{\nu}(\cdot; \ell, N)$ is zero if $|\nu| > |\ell|-N$ or if $\nu_j > \za_j$ for some $j$, so that the polynomial is
entirely zero on $V_{\ell,N}^d$ if $\nu \in \CH_{\ell,N}^d$ but $\nu \notin H_{\ell,N}^d$. 
\end{proof}

This theorem can be regarded as an extension of Theorem \ref{thm:factor1D} in one variable. In that theorem, 
the polynomial $\sQ_n(\cdot,\ell_1,\ell_2,N)$ in one variable is factored into two lower degree polynomials, 
one of which vanishes on $V_{\ell, N}$ so that the vanishing of $\sQ_n(\cdot,\ell_1,\ell_2,N)$ becomes obvious
from the factorization. An analogous factorization, however, no longer holds for $\sQ_\nu(\cdot; \ell, N)$ in 
higher dimension. In fact, what makes this theorem interesting lies in the existence of non-trivial polynomials that 
vanish on the set $V_{\ell,N}^d$, as illustrated by the following example. 

\begin{exam}\label{ex:example1}
Let $d=2$, $\ell_1 =6$, $\ell_2 = 4$, $\ell_3 = 4$, $N = 7$. Then $V_{\ell, N}^2$ and $H_{\ell,N}^2$ each 
contains 23 points; see Figure 1. 
\begin{figure}
\centering
 \includegraphics[width=2.3in]{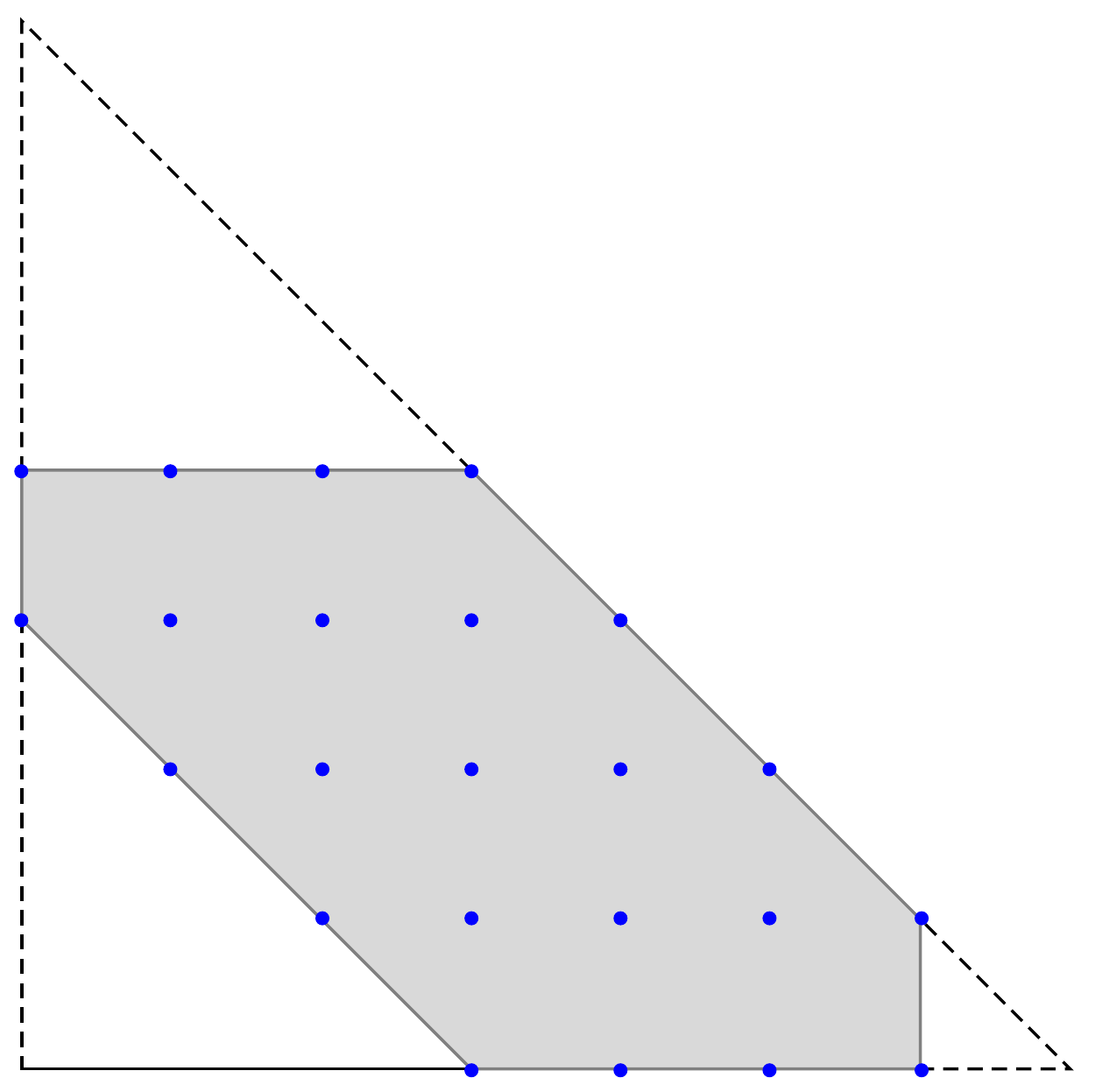}  \includegraphics[width=2.3in]{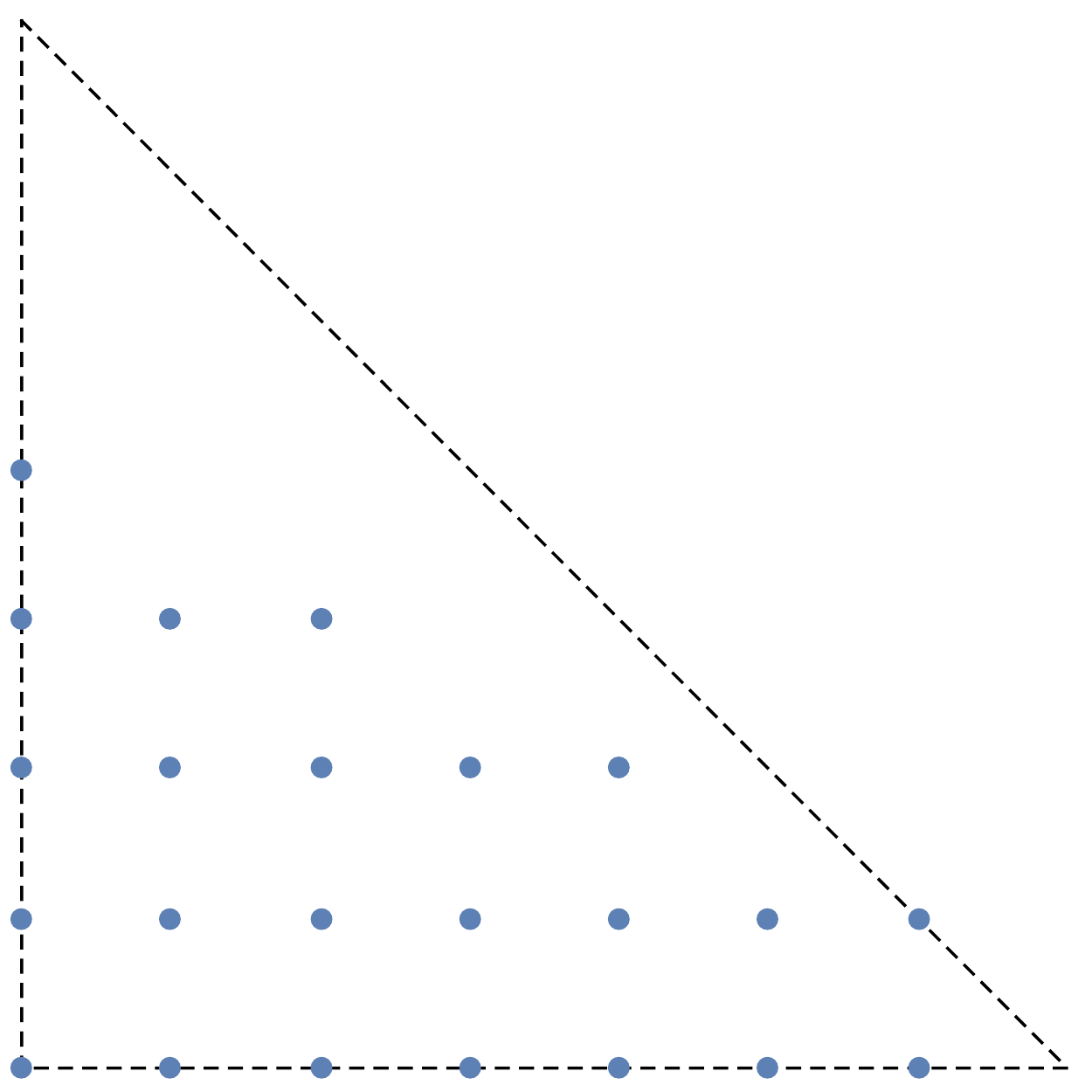}   
 \caption{$(\ell_1,\ell_2,\ell_3, N) = (6,4,4,7)$. Left: $V_{\ell,N}^2$. Right: $H_{\ell,N}^2$}
\end{figure}
The indices $\nu = (0,5)$ and $\nu = (3,3)$ are outside of $H_{\ell,N}^2$ and their corresponding polynomials 
\begin{align*}
   \sQ_{0,5} (x; \ell, N)  =  \, & (x_1-3)  \big(840 - 638 x_1 - 910 x_2 + 179 x_1^2 + 
   480 x_1 x_2 + 375 x_2^2 \\
      &\, \qquad \quad - 22 x_1^3  - 85 x_1^2 x_2 - 120 x_1 x_2^2  - 70 x_2^3 \\
      &\, \qquad \quad + x_1^4   + 5 x_1^3 x_2  + 
   10 x_1^2 x_2^2 + 10 x_1 x_2^3 + 5 x_2^4\big); \\
  \sQ_{3,3} (x; \ell, N)  = \, & -\frac{1}{48} (x_1-4) (x_1-3) (x_1-2) ( x_1 + 2 x_2 -7)  \\
     & \times (60 - 22 x_1 - 35 x_2 +  2 x_1^2  + 5 x_1 x_2 + 5 x_2^2),
\end{align*}
of degree 5 and 6, respectively, vanish on $V_{\ell,N}^2$; see Figure 2. 
\begin{figure}
\centering
 \includegraphics[width=2.3in]{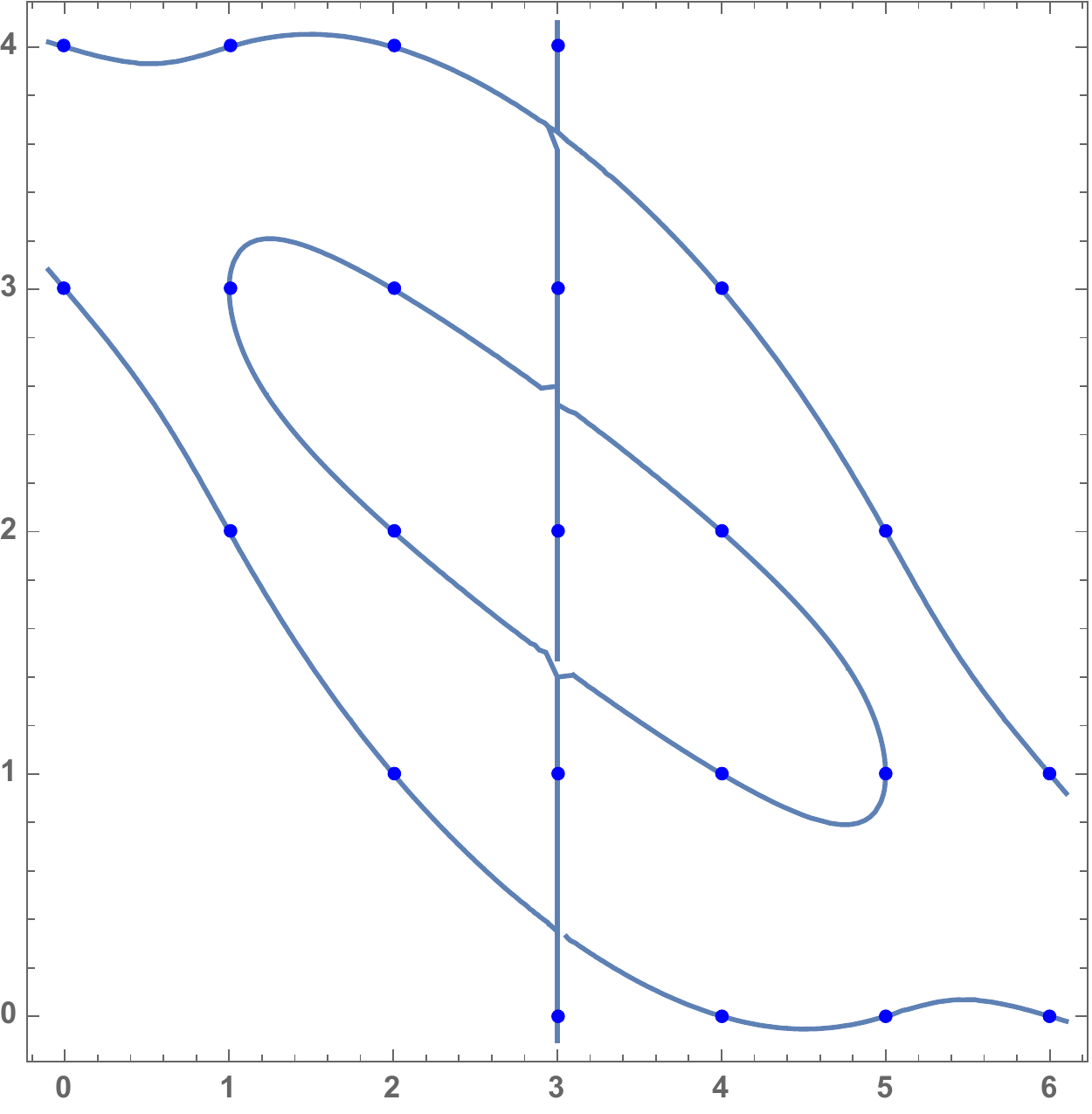}  \, \includegraphics[width=2.3in]{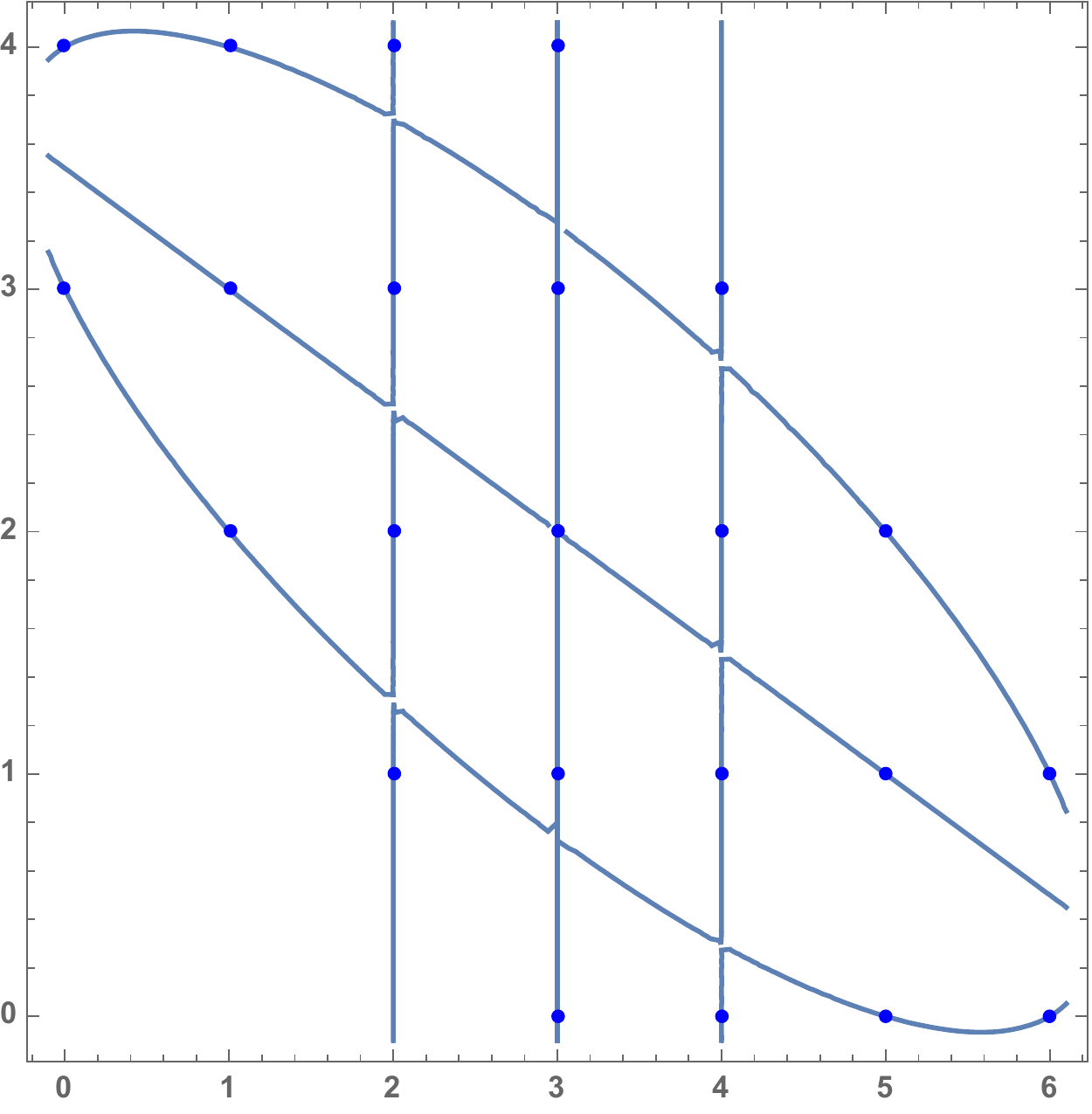}  
 \caption{$(\ell_1,\ell_2,\ell_3, N) = (6,4,4,7)$. Left: $ \sQ_{0,5} (x; \ell, N)$. Right: $ \sQ_{3,3} (x; \ell, N)$}
\end{figure}
\end{exam}

In the case of $\nu=(0,5)$, the polynomial $\sQ_{0,5}(x;\ell, N)$ contains an irreducible polynomial 
of degree $4$, thus its vanishing on $V_{\ell,N}^d$ is no longer obvious from the factorization, in 
contrast to Theorem \ref{thm:factor1D} in the one variable case. Observe that the polynomial of degree $4$
in $\sQ_{0,5}(x;\ell, N)$ and the irreducible polynomial of degree $2$ in $\sQ_{3,3} (x; \ell, N)$ have 
common zeros at 8 lattice points, which is the maximum possible predicted by the Bezout theorem. 

It is possible to be more specific about the factors in $\sQ_{\nu}(x;\ell, N)$ when $\nu$ is outside
of $H_{\ell,N}^d$. Let $N_j := N- |\xb_{j-1}|-|\bnu^{j+1}|$ so that $\sQ_{\nu_j}$ can be written as 
$\sQ_{\nu_j}(x_j;  \ell_j, \za_j, N_j)$. From $\nu_j \le \ell_j$, $1\le j \le d$, we 
obtain $|\bnu^{j+1}| \le |\bell^{j+1}| - \ell_{d+1}$, which implies, by \eqref{eq:l-condition}, that
$$
  \ell_j + \za_j  = \ell_j + |\bell^{j+1}| - 2 |\bnu^{j+1}| \ge \ell_j + \ell_{d+1} -|\bnu^{j+1}| \ge N- |\bnu^{j+1}|. 
$$
Hence, for $x \in V_{\ell,N}^d$, $ \ell_j + \za_j\ge N- |\bnu^{j+1}|- |\xb_{j-1}| = N_j$, which shows that the 
parameters of $\sQ_{\nu_j}$ in $\sQ_\nu$ satisfy \eqref{eq:ell-N}. Thus, if $N_j$ is a positive integer,
then $\sQ_{\nu_j}$ is a Hahn polynomial with negative integer parameters, so that Theorem \ref{thm:factor1D}
may apply if $\za_j < N_j$. 

\section{Factorization of Hahn polynomials of two variables with negative integer parameters}
\setcounter{equation}{0}

We examine the factorization of the Hahn polynomials of two variables more closely in this section. 
For $d =2$, the polynomial $\sQ_\nu (\cdot; \ell, N)$ is given by
\begin{align}\label{eq:Hahn2D}
 \sQ_\nu (x; \ell, N)  =   (-N)_{\nu_1}& \sQ_{\nu_1}(x_1;\ell_1,\ell_2+\ell_3- 2\nu_2, N-\nu_2) \\
      & \times  (-N+x_1)_{\nu_2}\sQ_{\nu_2}(x_2;\ell_2,\ell_3, N- x_1) \notag
\end{align}
and the index set $\nu \in H_{\ell,N}^2$ is given by 
\begin{align*}
    H_{\ell,N}^2 = \,& \big \{(\nu_1,\nu_2): \nu_1+\nu_2 \le N, \, \nu_1+\nu_2 \le  |\ell|-N, \\
       &  \quad \nu_1 \le \ell_1,\,  \nu_2 \le \ell_2, \, \nu_2\le \ell_3,  \, \nu_1+2\nu_2 \le \ell_2+\ell_3\big \}.
\end{align*}
It is easy to see that both $\sQ_{\nu_i}$, $i=1,2$, are Hahn polynomial with negative integer parameters. 
To determine their factorization for $\nu$ is outside of $H_{\ell,N}^2$, we need a precise description 
of $H_{\ell,N}^2$ using the height function $h_{\ell,N}$ introduced in \cite{IX20}. 

The height function measures the number of integer points on the vertical line $x_1 =\nu_1$ in $H_{\ell,N}^d$.
For $d = 2$, it is defined by  
$$
h_{\ell,N} (\nu_1):= \min\left(\ell_2,\ell_3,\left\lfloor \frac{\ell_2+\ell_3-\nu_1}{2}\right\rfloor,\ell_1+\ell_2+\ell_3-N-\nu_1,
N-\nu_1\right)+1.
$$
It follows immediately that the index set $H_{\ell,N}^2$ can be written as 
\begin{align*}
       H_{\ell,N}^2: = \{ (\nu_1,\nu_2):  0 \le \nu_1\le \ell_1, 0\le \nu_2 \le h_{\ell,N}(\nu_1) -1\}. 
\end{align*}
The value of the height function can be described more explicitly. 

\begin{lem}\label{lem:height}
The height function $h_{\ell, N}$ satisfies 
\begin{enumerate}[    \rm (i)]
\item if $0 \le \nu_1 \le |\ell_3-\ell_2|$, then $h_{\ell,N} (\nu_1) = \ell_2 \wedge \ell_3 +1$;
\item if $|\ell_3 - \ell_2| \le \nu_1 \le \ell_1 - \big |2N-|\ell| \big|$, then
$$
  h_{\ell,N}(\nu_1) = \ell_2 \wedge \ell_3  - \left \lfloor \frac{\nu_1-|\ell_3-\ell_2|+1}{2}\right \rfloor +1;
$$
\item if $\ell_1 - \big |2N-|\ell| \big| \le \nu_1 \le \ell_1$, then
$$
  h_{\ell,N}(\nu_1) = \begin{cases} |\ell| - N -\nu_1 +1, & |\ell| \le 2 N, \\  
        N - \nu_1 +1, & |\ell| > 2 N.\end{cases} 
$$
\end{enumerate}
\end{lem}

\begin{proof}
We can assume $\ell_3 \ge \ell_2$. The proof is a simple but tedious verification. The first two items have 
already been used in the proof of Lemma 4.3 in \cite{IX20}. We omit the details.
\end{proof}

To consider polynomials $\sQ_\nu(\cdot; \ell, N)$ with $\nu = (\nu_1,\nu_2)$ outside the domain $H_{\ell,N}^2$, 
we then assume $\nu_2 \ge h_{\ell,N}(\nu_1)$. Below we study $Q_\nu(\cdot; \ell, N)$ with 
$\nu_2 = h_{\ell,N}(\nu_1)$, that is, with index just outside $H_{\ell,N}^2$. For convenience, we shall
write 
\begin{equation} \label{eq:Hahn2D-2}
 \sQ_\nu (x; \ell, N)  =   (-N)_{\nu_1} \sQ_{\nu_1}(x_1; \wh \ell, \wh N)\, \sR_{\nu_2}(x;\ell_2,\ell_3, N),
\end{equation}
where $\wh\ell = (\ell_1,\ell_2+\ell_3-2\nu_2)$, $\wh N = N-\nu_2$, and  
\begin{equation} \label{eq:sRv2}
 \sR_{\nu_2}(x; \ell_2,\ell_3,N) := (-N +x_1)_{\nu_2} \sQ_{\nu_2}(x_2,\ell_2,\ell_3, N- x_1).
\end{equation}

We say that a polynomial of a variable $x$ splits if it splits in $\QQ$, i.e. if it can be written as 
$c \prod_j (x-a_j)$, $a_j \in \QQ$.  We examine $\sQ_{\nu_1}(x_1; \wh \ell, \wh N)$ first. Applying 
Theorem \ref{thm:factor1D}, this polynomial has linear factors if $\wh \ell_2 \le \wh N$ and 
$\deg_{\wh \ell, \wh N}< \nu_1 \le \min\{\ell_1, \wh N\}$, or more specifically, if  
\begin{equation}\label{eq:1st-factor}
 \ell_2+\ell_3-\nu_2  \le N, \quad  0  \le \min\{ \ell_1, N-\nu_2\} - \nu_1 < N- \ell_2-\ell_3 + \nu_2. 
\end{equation}

\begin{prop}\label{prop:sQv1}
Assume $\ell_3 \ge \ell_2$. Let $\nu_2 = h_{\ell,N}(\nu_1)$ and set 
$\sQ_{\nu_1}(x_1) = \sQ_{\nu_1}(x_1; \wh \ell, \wh N)$. Then  
\begin{enumerate}[(1)]
\item $\sQ_{\nu_1}$ is undefined if $|\ell | >  2 N$ and $\ell_1 - (|\ell |-2N)+1 \le \nu _1\le \ell_1$; 
\item $\sQ_{\nu_1}$ splits if $|\ell |\le 2 N$ and $1 \le \ell_3-\ell_2 -1 \le \nu_1 \le \ell_1$, or if 
$|\ell | >  2 N$ and $1 \le \nu_1 \le \ell_1 - (|\ell |-2N)$.
\end{enumerate} 
More precisely,  for $j =-1,0,1,\ldots$ and $\nu_1 = \ell_3-\ell_2 +j \le \ell_1 -\big| 2N - |\ell|\big|$,  
\begin{align} 
    \sQ_{\nu_1}(x_1) & = \frac{1}{(-N+\ell_2-i+1)_{\nu_1}} 
        \prod_{k=N- \ell_3-i+1}^{N- \ell_2 +i-1} (x_1-k),  \quad j = 2 i-1,  \label{eq:factorQv1A}\\
  \sQ_{\nu_1}(x_1) & = \frac{-N+\ell_3+i}{(-N+\ell_2-i+1)_{\nu_1}} (1-a_i x_1)
        \prod_{k=N-  \ell_3-i+1}^{N-\ell_2 +i-1} (x_1-k), \quad j = 2 i, \label{eq:factorQv1B}
 \end{align}
where $a_i=(\ell_1+\ell_2-\ell_3-2i)/((N- \ell_3-i)(N -\ell_2+i-1))$. Furthermore, if $|\ell| \le 2N$, then for 
$\nu_1 = \ell_1 + |\ell| -2 N +j$ with $j =1,2,\ldots, 2N -|\ell|$, 
\begin{align}\label{eq:factorQv1C}
   \sQ_{\nu_1}(x_1) & = \frac{1}{(- \ell_1)_{\nu_1}} \prod_{k= 2 N- |\ell| - j +1}^{\ell_1} (x_1-k).
\end{align}
\end{prop}

\begin{proof}
If $0 \le \nu_1 \le \ell_3-\ell_2$, then $h_{\ell,N} (\nu_1) = \ell_2 +1$. Hence, with $\nu_2 = \ell_2 +1$, it is 
easy to see that $\ell_2+\ell_3-\nu_2 = \ell_3-1\le N$ and $\min\{ \ell_1, N - \nu_2\} = N-\ell_2-1$ and 
$\deg_{\wh \ell, \wh N}= \ell_3-\ell_2-2$. 
Thus, \eqref{eq:1st-factor} holds if $\nu_1 > \ell_3 -\ell_2 -2$. Furthermore, if $\nu_1 = \ell_3-\ell_2 -1$, then
$\nu_1 = \deg_{\wh \ell, \wh N}+1$, so that we can apply Theorem \ref{thm:factor1D} with $m =0$ to factor
$\sQ_{\nu_1}$, which gives \eqref{eq:factorQv1A} for $i=0$, whereas if $\nu_1 = \ell_3-\ell_2$, then we can 
apply Theorem \ref{thm:factor1D} with $m =1$ to factor $\sQ_{\nu_1}$ and the factorization contains also 
$\sQ_1(x_1,\wh N - \wh \ell_2 -1, |N-\ell_1|-1, \wh N)$, a polynomial of degree 1 that gives the factor $1-a_0 x_1$,
which gives \eqref{eq:factorQv1B} for $i=0$

If $\ell_3 - \ell_2  < \nu_1 \le  \ell_1-\big| 2N - |\ell|\big|$, we write $\nu_1 = \ell_3 - \ell_2 + j$. By
Lemma \ref{lem:height},
$$
   h_{\ell,N} (\ell_3 - \ell_2 + j) = \ell_2 - \Big \lfloor \frac{j+1}{2}\Big \rfloor +1, 
        \quad 1 \le \Big \lfloor \frac{j+1}{2}\Big \rfloor  \le \min\{N-\ell_3, \ell_1+\ell_2 -N\}. 
$$ 
With $\nu_2 = h_{\ell,N} (\ell_3 - \ell_2 + j)$, we have  $\ell_2+\ell_3-\nu_2 \le \ell_3 + \lfloor \frac{j+1}{2} \rfloor -1 \le 
N -1 < N$. Moreover, $\min\{ \ell_1, N - \nu_2\} = N -  \ell_2 +  \lfloor \frac{j+1}{2}\rfloor -1$, so that 
$\deg_{\wh \ell, \wh N}  = \ell_3-\ell_2 +2 \lfloor \frac{j+1}{2}\rfloor -2$. Hence, \eqref{eq:1st-factor} holds for all 
$\ell_3 - \ell_2  < \nu_1 \le  \ell_1-\big| 2N - |\ell|\big|$. Furthermore, if $j =2i-1$, then $\nu_1 = \deg_{\wh \ell, \wh N}
+1$, so that we can apply Theorem \ref{thm:factor1D} with $m =0$ to factor $\sQ_{\nu_1}$, which gives \eqref{eq:factorQv1A} for $i>0$, whereas if $j = 2i$, then $\nu_1 = \deg_{\wh \ell, \wh N}+2$, we can then apply 
Theorem \ref{thm:factor1D} with $m =1$ to factor $\sQ_{\nu_1}$ and the factorization contains also
$\sQ_1(x_1,\wh N - \wh \ell_2 -1, |\wh N-\ell_1|-1, \wh N)$, which gives \eqref{eq:factorQv1B} for $i>0$.

If $\ell_1-\big| 2N - |\ell|\big|+1 \le \nu_1 \le \ell_1$,  we write $\nu_1 = \ell_1-\big| 2N - |\ell|\big| + j$ for
$j =1,2,\ldots,\big| 2N - |\ell|\big|$. Here we need to consider two cases. First, assume $|\ell| \le 2 N$. Then,
by Lemma \ref{lem:height},
$$
   h_{\ell,N} (\ell_1-\big| 2N - |\ell|\big| + j) = N-\ell_1 - j +1, \quad 1 \le j \le  2N - |\ell|. 
$$
With $\nu_2 = h_{\ell,N} (\ell_1-\big| 2N - |\ell|\big| + j)$, we have $\ell_2+\ell_3-\nu_2=  |\ell| -N+j -1 \le N -1 < N$
and $\deg_{\wh \ell, \wh N} = \ell_1 - (2N- |\ell|)+j-1$, so that $\nu_1 = \ell_1-\big| 2N - |\ell|\big| + j = 
\deg_{\wh \ell, \wh N}+1$ for all $j$. Thus, we can apply Theorem \ref{thm:factor1D} with $m =0$ to factor 
$\sQ_{\nu_1}$, which gives \eqref{eq:factorQv1C}. Next we assume $|\ell| > 2N$. Then, by
Lemma \ref{lem:height},
$$
 h_{\ell,N} (\ell_1-\big| 2N - |\ell|\big| + j) = \ell_2+\ell_3 - N- j +1, \quad 1 \le j \le  2N - |\ell|. 
$$
With $\nu_2 = h_{\ell,N} (\ell_1-\big| 2N - |\ell|\big| + j)$, a quick verification shows that 
$\nu_1+2\nu_2 -|\ell|-1 = -\ell_1-j+1$ and $N - \nu_2 = \nu_1-1$, which by \eqref{eq:HahnQN} leads to
$$
  \sQ_{\nu_1}(x_1) = {}_3 F_2 \left( \begin{matrix} -\nu_1, -\ell_1- j+1, -x\\
       -\ell_1, - \nu_1+1 \end{matrix}; 1\right).  
$$
Since $(-\nu_1+1)_{\nu_1} = 0$, this function is infinite for all $j \ge 1$. 
This completes the proof. 
\end{proof}

This proposition shows that the polynomials $\sQ_{\nu_1}$ is either undefined or splits with only 
possible exception when $2 \le \nu_1 \le |\ell_3 - \ell_2| -2$, and $|\ell_3 - \ell_2| \ge 4$, since the
factorizations are trivial for polynomials of degree $0$ or $1$. 
 
\begin{prop}\label{prop:4.3}
For $\nu_2 =  h_{\ell,N}(\nu_1)$, the polynomial $\sQ_{\nu}(x;\ell,N)$ is not well-defined if and only if 
\begin{enumerate} [  \rm(i)]
\item $\ell_3 \ge \ell_2$ and $\nu_1 \le \ell_3 - \ell_2$;
\item $|\ell | >  2 N$ and $2N+1-\ell_2-\ell_3\leq\nu_1$. 
\end{enumerate}
\end{prop}

\begin{proof}
We first assume $\ell_3 \ge \ell_2$. If $0\le \nu_1 \le \ell_3 - \ell_2$, 
then $\nu_2 = \ell_2 +1$ 
according to the proof of Proposition \ref{prop:sQv1}. For this $\nu_2$, it follows by \eqref{eq:HahnQN} that 
$$
 \sR_{\ell_2+1}(x; \ell_2,\ell_3,N) = (-N +x_1)_{\ell_2+1}\, {}_3 F_2 \left( \begin{matrix} -\ell_2 -1, -\ell_3, -x_2\\ 
       -\ell_2, - N+ x_1 \end{matrix}; 1\right),
$$
which shows that $\sR_{\ell_2+1}$ is undefined, since $(-\ell_2)_{\ell_2+1}  = 0$, if $\ell_3 > \ell_2$. The
case $\nu_1 > \ell_3-\ell_2$ follows readily from Proposition \ref{prop:sQv1}. Assume now $\ell_3 < \ell_2$
and $\nu_1 \le \ell_2-\ell_3$. Then $\nu_2=h_{\ell,N}(\nu_1)=\ell_3+1$ and 
\begin{align*}
 \sR_{\ell_3+1}(x; \ell_2,\ell_3,N) &= (-N +x_1)_{\ell_3+1}\, {}_3 F_2 \left( \begin{matrix} -\ell_3 -1, -\ell_2, -x_2\\ 
       -\ell_2, - N+ x_1 \end{matrix}; 1\right)\\
       &= (-N +x_1)_{\ell_3+1}\, {}_2 F_1 \left( \begin{matrix} -\ell_3 -1,  -x_2\\ 
        - N+ x_1 \end{matrix}; 1\right)= ( - N+ x_1+x_2)_{\ell_3+1},
\end{align*}
so that $\sR_{\ell_3+1}(x; \ell_2,\ell_3,N)$ is well defined and splits. Thus, the statement follows from Proposition \ref{prop:sQv1}.
\end{proof}

This shows that the non-trivial polynomials that vanishes on a larger set of lattice points, as seen in
Example \ref{ex:example1}, are $\sR_{\nu_2}(x; \ell,N)$. For $\nu_2 = h_{\ell,N}(\nu_1)$ and $\nu_1$ small,
it is possible to determine $\sR_{\nu_2}$ more explicitly. 

\begin{prop}
If $\ell_2 = \ell_3$, then 
$$
 \sR_{\ell_2}(x,\ell_2,\ell_2,N) = \frac{(-N+x_1+ x_2)_{\ell_2 +1} - (-1)^{\ell_2+1} (-x_2)_{\ell_2+1}}{-N+\ell_2+ x_1},
$$
which contains a linear factor $N-x_1-2x_2$ if $\ell_2$ is odd. Furthermore, 
$$
    \sR_{\ell_2+1}(x, \ell_2,\ell_2,N) = (-N+\ell_2+ x_1) \sR_{\ell_2}(x,\ell_2,\ell_2,N).
$$
\end{prop}

\begin{proof}
We consider the case $\sR_{\ell_2+1}$ first, which corresponds to $\nu_1 = 0$ and $\nu_2 = h_{\ell,N}(0) =\ell_2+1$. 
Thus, 
$$
 \sR_{\nu_2}(x; \ell_2,\ell_2,N) = (-N +x_1)_{\ell_2+1}\, {}_3 F_2 \left( \begin{matrix} -\ell_2 -1, -\ell_2, -x_2\\ 
       -\ell_2, - N+ x_1 \end{matrix}; 1\right).
$$
Taking as a limiting case, the ${}_3F_2$ function is given by
\begin{align*}
 \lim_{\ell_3\to \ell_2} {}_3 F_2 \left( \begin{matrix} -\ell_2 -1, -\ell_3, -x_2\\ 
       -\ell_2, - N+ x_1 \end{matrix}; 1\right)
      & =  \sum_{k=0}^{\ell_2} \frac{(-\ell_2-1)_k (-x_2)_k}{k! (-N+x_1)_k} \\
      & = {}_2 F_1 \left( \begin{matrix} -\ell_2 -1,  -x_2\\ 
        - N+ x_1 \end{matrix}; 1\right) - \frac{(-1)^{\ell_2+1} (-x_2)_{\ell_2+1}}{(-N+x_1)_{\ell_2+1}} \\
       &= \frac{(-N+x_1+x_2)_{\ell_2+1}}{(-N+x_1)_{\ell_2+1}} - \frac{(-1)^{\ell_2+1} (-x_2)_{\ell_2+1}}{(-N+x_1)_{\ell_2+1}},
\end{align*}
where the last step follows from Chu-Vandermonde identity. Consequently, 
\begin{align*}
 \sR_{\ell_2+1}(x; \ell_2,\ell_2,N)   =  (-N+x_1+x_2)_{\ell_2 +1} - (-1)^{\ell_2+1} (-x_2)_{\ell_2+1}. 
\end{align*}
If $\ell_2$ is odd, then the righthand side becomes zero if $x_1 = N - 2 x_2$, which shows that this
polynomial contains a factor $N-x_1-2x_2$. Furthermore, rewriting the first Pochhammer symbol in the
righthand side, we obtain 
$$
     \sR_{\ell_2+1}(x; \ell_2,\ell_2,N) = (-1)^{\ell_2+1} \big( (- x_2 -x_1+N-\ell_2)_{\ell_2 +1} - (-x_2)_{\ell_2+1}\big),
$$
which is zero when $x_1 = N-\ell_2$, so that it also contains a factor $N-\ell_2-x_1$. 

The case $\sR_{\ell_2}$ corresponds to $\nu_1 = 1$ and $\nu_2 = \ell_2$. We end up 
with the same ${}_3F_2$ function, 
so that, using $(-N+x_1)_{\ell_2 +1} =(-N+\ell_2+x_1)(-N+x_1)_{\ell_2}$, we obtain
$$
  \sR_{\ell_2+1}(x; \ell_2,\ell_2,N) = (-N+\ell_2+x_1)  \sR_{\ell_2}(x; \ell_2,\ell_2,N). 
$$
Since $ \sR_{\ell+1}(x; \ell_2,\ell_2,N)$ contains a factor $-N+\ell_2+x_1$, this completes the proof. 
\end{proof}
 
Since the first factor of $\sQ_{\nu}(x;\ell,N)$ when $\nu_1=1$, $\nu_2=\ell_2$ and $\ell_3=\ell_2$
is 
$$
  \sQ_1(x_1; \ell_1, 0, N-\ell_2) = \frac{-N+\ell_2 + x_1}{-N + \ell_2}, 
$$ 
we conclude that $\sR_{\ell_2}(x; \ell_2, \ell_2, N)$ is a polynomial of degree $\ell_2$ that
vanishes on $V_{\ell,N}^2 \setminus \{ (N-\ell_2, j): 0 \le j \le \ell_2\}$. This is the polynomial of 
degree 4 in $\sQ_{0,5}(\cdot; \ell,N)$ in Example \ref{ex:example1}. 
 
If $\ell_3 > \ell_2$, then the smallest suitable $\nu_1$ is $\ell_3-\ell_2 +1$, for which $\nu_2 = \ell_2$. 
Consequently, we obtain 
\begin{align*}
 \sR_{\ell_2}(x; \ell_2,\ell_3,N) \, & = (-N +x_1)_{\ell_2}\, {}_3 F_2 \left( \begin{matrix} -\ell_2, -\ell_3-1, -x_2\\ 
       -\ell_2, - N+ x_1 \end{matrix}; 1\right) \\
       &= (-N +x_1)_{\ell_2} \sum_{k=0}^{\ell_2} \frac{(-\ell_3-1)_k (-x_2)_k}{k! (-N+x_1)_k}. 
\end{align*} 
If $\ell_3 - \ell_2$ is small, then the last sum can be made more explicit by writing  it in terms of 
${}_2F_1$, by adding and subtracting a few terms, and then using the Chu-Vandermonde identity, 
which gives 
\begin{align*}
  \frac{\sR_{\ell_2}(x; \ell_2,\ell_3,N)}{(- N+x_1)_{\ell_2}} = 
      \frac{ (-N+x_1+x_2)_{\ell_3+1}}{(-N+x_1)_{\ell_3+1}}
 - \sum_{k=\ell_2+1}^{\ell_3+1} \frac{(-\ell_3-1)_k (-x_2)_k}{k! (-N+x_1)_k}.
\end{align*}
By \eqref{eq:factorQv1A} and \eqref{eq:Hahn2D-2}, the polynomial $\sR_{\ell_2}$ must 
vanishes if $x_1 > N-\ell_2$ or $x_1 < N-\ell_3$, which however is not obvious from the formula. 

The simplest case is when $\ell_3 = \ell_2 +1$, for which $\nu_1 = 2$. We then obtain
\begin{align*}
    \sR_{\ell_2}(x; \ell_2,\ell_2+1,N) \, &= \frac{(-N+x_1+x_2)_{\ell_2+2}}{(-N+\ell_2+x_1)_2} \\
      & - \frac{(-1)^{\ell_2+1} (\ell_2+2)(-x_2)_{\ell_2+1}}{-N+\ell_2+x_1} 
     -\frac{(-1)^{\ell_2+2}  (-x_2)_{\ell_2+2}}{(-N+\ell_2+x_1)_2}.
\end{align*}
Since the first factor of $\sQ_{\nu}(x;\ell,N)$ in this case is a quadratic polynomial
$$
  \sQ_2(x_1; \ell_1, 1, N-\ell_2) = \frac{(N-\ell_2 - x_1)(N-\ell_2-1-x_1)}{(N - \ell_2)(N-\ell_2-1)} 
$$
that vanishes when $x_1 = N-\ell_2$ and $N-\ell_2-1$, it follows that the polynomial 
$\sR_{\ell_2}(x; \ell_2, \ell_2+1, N)$ is a polynomial of degree $\ell_2$ that vanishes on 
$$
V_{\ell,N}^2 \setminus (\{ (N-\ell_2, j): 0 \le j \le \ell_2\} \cup \{ (N-\ell_2-1, j): 0 \le j \le \ell_2\}).
$$ 
This is not, however, obvious from the explicit formula of $\sR_{\ell_2}(x; \ell_2, \ell_2+1, N)$ given above. 

We also observe that $V_{\ell,N}^2$ depends on $\ell_1$, whereas $\sR_{\ell_2}(x; \ell_2, \ell_2+1, N)$
is independent of $\ell_1$. Consequently, the polynomial $\sR_{\ell_2}(x; \ell_2, \ell_2+1, N)$ vanishes on a 
large number of lattice points. 

We end this subsection by making a conjecture on the irreducibility of the polynomial $\sR_{\nu_2}$, which we 
state more generally for the polynomial 
\begin{align*}
R_n(x; \ell_1,\ell_2, y) &:= (-y)_n  \sQ_n(x; \ell_1,\ell_2, y)\nonumber\\
& = (-y)_n\sum_{k=0}^{n} \frac{(-n)_k(n-\ell_1-\ell_2-1)_k(-x)_k}{(-\ell_1)_k\,k!\,(-y)_{k}}\nonumber\\
& = \sum_{k=0}^{n} \frac{(-n)_k(n-\ell_1-\ell_2-1)_k(-x)_k(-y+k)_{n-k}}{(-\ell_1)_k\,k!}.
\end{align*}
This is a well-defined polynomial in $\QQ[x,y]$ when $n\leq \ell_1 \wedge \ell_2$.

\begin{conj} \label{Conj}
For $n\in\NN$ and $\ell_1,\ell_2\in\NN$ such that $n\leq \ell_1 \wedge \ell_2$, the polynomial 
$R_n(x; \ell_1,\ell_2, y) $ is irreducible unless $\ell_1=\ell_2$ and $n$ is odd, in which case 
$R_n(x; \ell_1,\ell_2, y) $ is the product of $(y-2x)$ and an irreducible polynomial of degree $n-1$.
\end{conj}

For all $n \le \ell_1\wedge \ell_2$, the identity 
$$
R_n(x; \ell_1,\ell_2, y)=(-1)^n \frac{(-\ell_2)_n}{(-\ell_1)_n} R_n(y-x; \ell_2,\ell_1, y)
$$
holds, which shows that that $R_n(x; \ell_1,\ell_2, y)$ has a factor $y-2x$ if $\ell_1=\ell_2$ and $n$ is odd.
The conjecture states that the polynomial is irreducible apart from this trivial factor. 

The conjecture can be  naturally viewed as a two-dimensional analog of irreducibility results for classical orthogonal polynomials, which have a long history. Indeed, a famous result by Hilbert asserts that there exist  irreducible polynomials of every degree $n$ over $\QQ$ having the largest possible Galois group $S_n$. While Hilbert’s proof was nonconstructive, Schur provided a rather explicit example by proving that the $n$-th Laguerre polynomial is irreducible and has Galois group $S_n$ over $\QQ$. In the early 50’s, Grosswald conjectured the irreducibility of the Bessel polynomials, which was proved in \cite{FT}. 

It is worth noting that, in general, irreducibility results for polynomials of two variables are easier to prove, since $\QQ[x,y]=\QQ[x][y]$ and we can try to use the Eisenstein criterion for the unique factorization domain $\QQ[x]$. This is sometimes the ``easy step" in the classical one-dimensional results, see for instance Proposition 3.1 in \cite{CHS}. However, these arguments do not work in our case, since the expansion is in Pochhammer terms, and not in ordinary powers of $x$ or $y$. 

\section{Generating function}
\setcounter{equation}{0} 
 
For Hahn polynomials on lattice points in the simplex, the generating function is a useful tool and plays
an essential role in \cite{IX17, X15}. In view of the generating function \eqref{eq:generatingHahnQ}, one
would expect that there is a generating function for $\sQ_\nu(\cdot;\ell,N)$ given in terms of the Jacobi polynomial 
on the simplex with negative integer parameters. Using the normalized Hahn polynomials
$$
\sH_\nu(\a; \ell, N) = H_\nu(\a; - \ell - \one, N), \qquad \a \in \NN_0^{d+1},
$$
see \eqref{eq:HahnH}, one such extension holds straightforwardly. 

\begin{prop}
Let $\ell_i$ and $N$ are positive integers, so that $\ell_i \le N$ and \eqref{eq:l-condition} holds. 
For $\nu \in \NN_0^d$, define 
$$
  G_\nu(x):= G_\nu^\ell(x) =  \prod_{j=1}^d \left(1-|\xb_{j-1}| \right)^{\nu_j}  G_{\nu_j}^{(\za_j,\ell_j)}\left 
  (\frac{2x_j}{1-|\xb_{j-1}|} -1\right), 
$$
where $\za_j$ is defined in \eqref{eq:saj}. Then, for $\nu\in H_{\ell,N}^d$, the Hahn 
polynomials $\sH_\nu(\cdot; \ell, N)$ satisfy 
\begin{equation}\label{eq:Hahngenfunc2}
 G_{\nu,N}(y) = |y|^N G_\nu\Big ( \f {y'} {|y|} \Big)
   =   \sum_{|\alpha| = N} \frac{N!}{\alpha!}\sH_\nu(\alpha; \ell,N)y^\alpha.
\end{equation}
\end{prop}

This follows from the identity \eqref{Hahngenfunc}, which is a finite sum for $\nu \in H_{\ell,N}^d$, by
setting $\k_i = - \ell_i -1$. The sum in the right-hand side is over all lattice points in $V_N^d$, which 
contains $V_{\ell,N}^d$ as a subset. One may ask if this is a multidimensional analog of the generating
function in Proposition \ref{prop:GenF1d} so that the right-hand side of \eqref{eq:Hahngenfunc2}
is summed over lattice points in $V_{\ell,N}^d$. However, this does not look to be possible, since for the factor
$\sQ_{\nu_j}(x_j;  \ell_j, \za_j, \wh N)$ of $\sH_\nu(\cdot; \ell,N)$, where $\wh N = N- |\xb_{j-1}|-|\bnu^{j+1}|$
and $j \ge 2$, both $\ell_j \wedge \wh N$ and $\za_j \wedge \wh N$ depend on $|\xb_{j-1}|$ for 
$x \in V_{\ell,N}^d$.

The generating function  \eqref{Hahngenfunc} is used to derive properties of the Hahn polynomials on
$V_N^d$ in \cite{X15}. Some of those properties remain valid when the parameters are negative integers. 
Of particular interests are those on the reproducing kernel of the Hahn polynomials of degree $n$, defined by 
$$
  P_n(W_{\k,N}; x,y) =  \sum_{|\nu|=n} \frac{Q_\nu(x;\k, N) Q_\nu(y;\k, N)}{B_\nu(\k,N)}, \qquad 0 \le n \le N.
$$
For the Hahn polynomials with negative integer parameters, the corresponding kernel is defined by 
$$
  \sP_n (\sH_{\ell,N}; x,y) = \sum_{\substack{|\nu|=n \\ \nu \in H_{\ell,N}^d}} 
      \frac{\sQ_\nu(x;\ell, N)\sQ_\nu(y;\ell, N)}{\sB_\nu(\ell,N)}, \qquad 0 \le n \le N.
$$
The space $\CV_n^d(\sH_{\ell,N})$ is of the dimension $\# \{\nu \in H_{\ell,N}^d: |\nu| =n\}$, which is less 
than $\# \{\nu \in \NN_0^d: |\nu| =n\}$ for $n$ large. However, let 
$$
\ell_{\rm min} :=\min \{\ell_i: 1 \le i \le d+1\}.
$$

\begin{lem}\label{eq:indexl_min}
The set $\{\nu \in H_{\ell,N}^d: |\nu| = n\}$ coincides with $\{\nu \in \NN_0^d: |\nu| =n\}$ if and only if 
$0 \le n \le \ell_{\rm min}$. 
\end{lem}

\begin{proof}
If $0 \le n \le \ell_{\rm min}$ and $|\nu| = n$, then $\nu_j \le n \le \ell_{\rm min}$ so that 
$\nu_j \le \ell_j$ follows trivially. Moreover, $|\nu^j| + |\nu^{j+1}| \le 2 n \le |\ell^{j+1}|$ holds if 
$j \le d-1$, which implies $\nu_j \le \za_j$ for $1 \le j \le d-1$, whereas $\za_d = \ell_{d+1}$, so that
$\nu_d \le \za_d$ holds as well. This shows that $\{\nu \in \NN_0^d: |\nu| =n\} \subset
\{\nu \in H_{\ell,N}^d: |\nu| = n\}$, while the other direction of inclusion is trivial. Hence, the two
sets are equal. If $n > \ell_{\rm min}$ and assume, say, $\ell_{\rm min} = \ell_1$, then the element
$\nu = (n,0,\ldots,0)$ does not belong to $\{\nu \in H_{\ell,N}^d: |\nu| = n\}$ as $\nu_1 = n > \ell_1$. 
This completes the proof. 
\end{proof} 

For $0 \le n \le \ell_{\rm min}$, we can then write the sum of $\sP_n (\sH_{\ell,N};\cdot,\cdot)$ as over 
$\{\nu: |\nu| = n\}$. In this way, it is easy to see that the kernel agrees with $P_n(W_{\k,N}; \cdot,\cdot)$ 
when $\k_j = -\ell_j-1$. In particular, by \cite[Theorem 4.3]{X15}, we can rewrite $\sP_n(\sH_{\ell,N})$ 
in terms of elementary function
$$
  \CE_k (x,y; \ell) :=  \sum_{\substack{|\g|= k \\ \g_i \le \ell_i}}   \frac{(-X)_{\g}(-Y)_\g}{(-\ell)_\g \g!}, \quad  
       x,y \in V_{\ell,N}^d,
$$         
where $X= (x,N-|x|)$, $Y= (y, N-|y|)$ and $\g \in \NN_0^{d+1}$, and we have used the notation 
$(a)_\g =(a_1)_{\g_1} \cdots (a_{d+1})_{\g_{d+1}}$ for $a \in \RR^{d+1}$. 

\begin{prop} \label{thm:PnH}
Let $\ell_i$ be nonnegative integers, $\ell_i \le N$ and $\ell_i+\ell_j \ge N$. Then, for $x , y \in V_{\ell,N}^d$ 
and $0 \le n \le \ell_{\rm min}$, 
\begin{align}\label{PnH-closed}
     \sP_n(\sH_{\ell, N}; x,y) = & \frac{(-N)_n (-|\ell|)_N (-|\ell|)_n (|\ell| +1-2n) }{ n!(-|\ell|)_{N+n}(|\ell|+1-n)} \\
    & \times  \sum_{k=0}^n \frac{(-n)_k (n-1-|\ell|)_k}{(-N)_k(-N)_k} \CE_{k} (x,y;\ell). \notag
\end{align}
\end{prop}

If $n > \ell_{\rm min}$, then $\{\nu: |\nu| = n\}$ contains $\nu$ outside of $H_{\ell,N}^d$, so that some 
of $\sQ_{\nu}(\cdot; \ell, N)$ with $|\nu| = n$ vanishes on $V_{\ell,N}^d$ by Theorem \ref{thm:CH_ellN}. One 
may ask if it is possible to include such polynomials in the summation of $\sP_n(\sH_{\ell, N}; \cdot,\cdot)$
so that some version of \eqref{PnH-closed} can be deduced from that of $P_n(W_{\k,N}; \cdot,\cdot)$. 
The answer, however, is negative since the norm of such polynomials, $B_\nu(\ell,N)$, is necessarily 
zero and, as a result, no such terms can be included in the sum of $\sP_n(\sH_{\ell, N}; \cdot,\cdot)$. 

There is an exceptional case when $d =2$, for which Hahn polynomials of negative integer parameters of
degree $n$ have full range for all suitable $n$. This is the case when $\ell_1 = \ell_2  = \ell_3 =\ell$ and 
$N  = 2 \ell$, so that the set $V_{\ell, N}^2$ consists of lattice points in a regular triangle, whereas the set
$H_{\ell,N}^2$ is given below. 

\begin{prop}\label{prop:d=2triangle}
Let $\ell$ be a positive integer. If $d =2$, $\ell_i = \ell$ for $ 1 \le i \le 3$ and $N = 2 \ell$, then 
$$
     H_{\ell, N}^2  = \{\nu \in \NN_0^2:  |\nu| \le \ell\}.
$$
\end{prop}

\begin{proof}
Since $N = 2 \ell$ and $|\ell |  = 3 \ell$, we also have $\za_1 = 2\ell - 2 \nu_2$ and $\za_2 = \ell$. Hence,
$$
   H_{\ell,N}^2 =  \big \{(\nu_1,\nu_2): |\nu| \le 2 \ell, \, |\nu| \le \ell, \, \nu_1 \le \ell,  \,  \nu_2 \le \ell,
        \nu_1+2\nu_2 \le 2 \ell\big \}
$$
by the definition of $H_{\ell,N}^2$. Clearly, $|\nu|\le \ell$ implies $\nu_i \le \ell$ and it also implies
$\nu_1 + 2 \nu_2 = \nu_1 + |\nu| \le 2 \ell$, which proves the statement. 
\end{proof}

For the Hahn polynomials $Q_\nu(\cdot; \k,N)$ on $V_N^d$, the closed form expression of the reproducing
kernel is used to prove that the Poisson kernel 
$$
  \Phi_r(W_{\k,N};x,y): = \sum_{n=0}^{N} \sP_n(W_{\k,N};x,y)  r^n, \qquad 0 \le r \le 1
$$
is nonnegative for all $x,y \in V_N^d$ and $0 \le r \le 1$, see \cite{X15}. The proof relies on the fact that the
corresponding function 
$$
  \CE_k (x,y; \k) :=  \sum_{\substack{|\g|= k}}   \frac{(-X)_{\g}(-Y)_\g}{(\k+1)_\g \g!}, 
 $$    
for $\k_i > -1$ is nonnegative for $x,y \in V_N^d$. 

For the Hahn polynomials with negative integer parameters, we could define the Poisson kernel analogously 
as 
$$
  \Phi_r(\sH_{\ell,N};x,y): = \sum_{n=0}^{M} \sP_n(\sH_{\ell,N};x,y)  r^n, \qquad 0 \le r \le 1,
$$ 
where $M \le N$ is the highest degree $M = \max\{ |\nu|: \nu \in H_{\ell,N}^d \}$. This kernel, however,
is no longer nonnegative, as shown by examples of small $\ell$ and $N$. One of the reasons that the
proof in \cite{X15} fails is that $\CE_k(x,y;\ell)$ is of the sign $(-1)^k$ instead of nonnegative. Furthermore, 
for $M > \ell_{\min}$, we could question if the definition $\Phi_r(\sH_{\ell,N};\cdot,\cdot)$ reflects the symmetry
of $\sH_{\ell,N}^d$. In this regard, we may examine the special case when $d=2$, $\ell_i = \ell$ and $N = 2 \ell$ 
in Proposition \ref{prop:d=2triangle} closely. 

\begin{prop}
Let $d=2$, $\ell_i = \ell$, $1 \le i \le 3$, and $N = 2 \ell$. Then, for all $x_2$,
\begin{align*}
 \Phi_r\big (\sH_{\ell,N}; \big((x_1, x_2), (0,\ell) \big)\big) 
   = \sum_{n=0}^\ell  \frac{(-1)^n (-3 \ell-1)_n (-3 \ell)_{2n}}{n! (-3 \ell -1)_{2n}}  \sQ_{n}(x_1; \ell, 2 \ell, 2 \ell) r^n
\end{align*}
\end{prop} 

\begin{proof}
Using the explicit formula of $\sQ_{\nu}(\cdot; \ell, N)$ in \eqref{eq:Hahn2D}, it is easy to verify that 
$$
   \sQ_{\nu_1,\nu_2} ( (0,\ell); \ell, 2 \ell) = (-2 \ell)_{\nu_1} \delta_{\nu_2,0}
$$
by the Chu-Vandermonde identity. For $\nu_2 =0$, we have
$$
     \sQ_{\nu_1,0} ( x; \ell, 2 \ell)  = (-2\ell)_{\nu_1} \sQ_{\nu_1}(x_1; \ell, 2 \ell, 2 \ell). 
$$
Furthermore, by \eqref{eq:Bnu}, the norm of $\sQ_{\nu_1,0}$ is equal to
$$
B_{\nu_1,0}(\ell, 2\ell) = \frac{(-1)^{\nu_1} \nu_1 ! (-2\ell)_{\nu_1}^2 (-3 \ell -1)_{2\nu_1}}
    {(-3 \ell-1)_{\nu_1} (-3 \ell)_{2\nu_1}}.
$$
Together, these identities lead to 
\begin{align*}
  \sP_n \big(\sH_{\ell,2\ell}; (\ell, x_2),( 0,\ell) \big) \, & = \sum_{\nu_1 =0}^n \frac{ \sQ_{\nu_1, n- \nu_1} 
      ( (0,\ell); \ell, 2 \ell)  \sQ_{\nu_1, n-\nu_1} (x; \ell, 2 \ell) }{B_{\nu_1,n-\nu_1}(\ell, 2\ell)}  \\
     & = \frac{(-1)^n (-3 \ell-1)_n (-3 \ell)_{2n}}{n! (-3 \ell -1)_{2n}}  \sQ_{n}(x_1; \ell, 2 \ell, 2 \ell). 
\end{align*}
from which the formula for $\Phi_r\big(\sH_{\ell,N}; (\ell, x_2), (0,\ell) \big)$ follows immediately. 
\end{proof}

This gives an explicit expression for the Poisson kernel $\Phi_r(\sH_{\ell,2\ell}; \cdot,\cdot)$, which is a polynomial 
of degree $\ell$ in the variable $r$, at the points $(x_1, x_2)$ and $(0,\ell)$ in $V_{\ell,2\ell}^2$, which depends
on $x_1$ but not $x_2$. Numerical experiment shows that this polynomial changes sign once on $[0,1]$ 
when $\ell$ is even and $x = (\ell, x_2)$ and, in some cases, when $\ell$ is odd  and $x$ is an even integer.

\section{Bispectral properties}
\setcounter{equation}{0}
The Hahn polynomials for the hypergeometric distribution can be characterized as common eigenfunctions of two families of commutative algebras of difference operators: one acting on the variables $x_1,\dots,x_d$, and another one acting on the indices $\nu_1,\dots,\nu_d$. These operators can be linked to mutually commuting symmetries of a discrete extension of the generic quantum superintegrable system on the sphere \cite{DGVV,I17,IX20,KMP}. 
We discuss these families of operators in the next subsections.

In this section, let $\wh \sQ_\nu(x;\ell, N)$ denote the Hahn polynomials normalized as follows
\begin{align} \label{eq:Hahn-n}
 \wh \sQ_\nu(x;\ell, N) = \frac{1}{(-N)_{|\nu|}}\prod_{j=1}^d  &
     (-N+|\xb_{j-1}|+|\bnu^{j+1}|)_{\nu_j} \\
      & \times \sQ_{\nu_j}\left(x_j;  \ell_j, \za_j, N- |\xb_{j-1}|-|\bnu^{j+1}|\right). \notag
\end{align}

The polynomial $\wh \sQ_\nu$ differs from $\sQ_\nu$ by an an extra factor $(-N)_{|\nu|}$ in the denominator. This factor makes it possible to write the recurrence operators $\CL_d^\nu$, defined in 
\eqref{eq:operatordn}, in its relatively simple form. 

\subsection{Spectral equations in the variables.}
We denote by $\{e_1,e_2,\dots,e_d\}$ the standard basis for $\RR^d$, and by $E_{x_i}$ and $E_{x_i}^{-1}$ the shift operators acting on a function $f(x)$  as follows
\begin{align*}
E_{x_i}f(x)=f(x+e_i) \quad \hbox{and}\quad E_{x_i}^{-1}f(x)=f(x-e_i).
\end{align*}
The operator 
\begin{align}
\CL^{x}_{d}&:=\CL^{x}_{d}(x;\ell;N)=\sum_{1\leq i\neq j\leq d}x_j(x_i-\ell_i)(E_iE_j^{-1}-1)\label{eq:operatordx}\\
&+\sum_{i=1}^d(x_i-\ell_i)(N-|x|)(E_i-1)+ \sum_{i=1}^dx_i(N-|x|-\ell_{d+1})(E_i^{-1}-1),\nonumber
\end{align}
is self-adjoint with respect to the hypergeometric distribution and acts diagonally on the basis of polynomials in 
\eqref{eq:Hahn-n} with eigenvalue $-|\nu|(|\nu|-|\ell|-1)$ which depends only on the total degree $|\nu|$ of the 
polynomial $\wh \sQ_\nu(x;\ell, N)$, see \cite[Section 5]{IX17}. Fix now $k<d$ and note now that, up to a factor 
independent of $x_1,\dots,x_{d-k}$, the product of the last $k$-terms in \eqref{eq:Hahn-n} can be regarded as 
a Hahn polynomial in the variables $\xt=\xb^{d-k+1}=(x_{d-k+1},\dots,x_d)$, with indices 
$\tilde{\nu}=\bnu^{d-k+1}=(\nu_{d-k+1},\dots,\nu_d)$, and parameters 
$\tilde{\ell}=\bell^{d-k+1}=(\ell_{d-k+1},\dots,\ell_{d+1})$, $\tilde{N}=N-|\xb_{d-k}|$. Therefore, if we set 
$$
\CL^{x}_{k}:=\CL^{x}_{k}(\xb^{d-k+1}; \bell^{d-k+1}; N-|\xb_{d-k}|), \qquad\text{ for }k=1,\dots,d-1,
$$
we see that polynomials in \eqref{eq:Hahn-n} will be eigenfunctions of the operators $\CL^{x}_{1},\dots, \CL^{x}_{d}$ 
and satisfy the spectral equations, for $k=1,\dots,d$, 
\begin{equation}\label{eq:spx}
\CL^{x}_{k} \wh \sQ_\nu(x;\ell, N)=-|\bnu^{d-k+1}| (|\bnu^{d-k+1}|-|\bell^{d-k+1}|-1) \wh \sQ_\nu(x;\ell, N). 
\end{equation}
From these equations it follows easily that the operators $\CL^{x}_{1},\dots, \CL^{x}_{d}$ commute with each other. 
They generate a Gaudin subalgebra for a representation of the Kohno-Drinfeld algebra associated with the 
hypergeometric distribution. The operators in the larger algebra can be regarded as  symmetries, or integrals 
of motion, for a discrete extension of the generic quantum superintegrable system on the sphere, see \cite{IX20}.

\subsection{Spectral equations in the indices.} In this section, we will use the shift operators $E_{\nu_i}$ 
and $E_{\nu_i}^{-1}$ acting on a function $f_{\nu}$  as follows
\begin{align*}
E_{\nu_i}f_{\nu}=f_{\nu+e_i} \quad \hbox{and}\quad E_{\nu_i}^{-1}f_{\nu}=f_{\nu-e_i}.
\end{align*}

For $j,k\in\{0,\pm1\}^2$ and $i\in\{1,\dots,d\}$, we define $B_{i}^{j,k}$
as follows
\begin{align*}
&B_i^{0,0}=|\bnu^{i}|(|\bnu^{i}| -|\bell^{i}|-1)+|\bnu^{i+1}|(|\bnu^{i+1}| -|\bell^{i+1}|-1)
+ \frac{|\bell^{i+1}|(|\bell^{i}|+2)}{2},\\
&B_i^{0,1}=-\nu_i(\nu_i+2|\bnu^{i+1}|-|\bell^{i}|-1),\\
&B_i^{0,-1}=(\ell_i-\nu_i)(\nu_i+2|\bnu^{i+1}|-|\bell^{i+1}|-1),\\
&B_i^{1,0}=(\ell_i-\nu_i)(\nu_i+2|\bnu^{i+1}|-|\bell^{i}|-1),\\
&B_i^{-1,0}=-\nu_i(\nu_i+2|\bnu^{i+1}|-|\bell^{i+1}|-1),\\
&B_i^{1,1}=(\nu_i+2|\bnu^{i+1}|-|\bell^{i}|-1)(\nu_i+2|\bnu^{i+1}|-|\bell^{i}|),\\
&B_i^{-1,1}=\nu_i (\nu_i-1),\\
&B_i^{1,-1}=(\ell_i-\nu_i) (\ell_i-\nu_i-1),\\
&B_i^{-1,-1}=(\nu_i+2|\bnu^{i+1}|-|\bell^{i+1}|-1)(\nu_i+2|\bnu^{i+1}|-|\bell^{i+1}|-2).
\end{align*}
We extend the formulas above and define $B_{0}^{0,k}$ for $k\in \{0,\pm1\}$ as follows
\begin{equation*}
B_0^{0,0}=-N+ |\ell|/2, \quad B_0^{0,1}=N-|\nu|, \quad B_0^{0,-1}=|\nu|-|\ell|+N-1.
\end{equation*}
Next, for $i\in \{1,\dots, d\}$ we define 
\begin{align*}
&b_i^{0}=\frac{(2|\bnu^{i}|-|\bell^{i}|)(2|\bnu^{i}|-|\bell^{i}|-2)}{2},\\
&b_i^{1}=(2|\bnu^{i}|-|\bell^{i}|)(2|\bnu^{i}|-|\bell^{i}|-1),\\
&b_i^{-1}=(2|\bnu^{i}|-|\bell^{i}|-2)(2|\bnu^{i}|-|\bell^{i}|-1),
\end{align*}
and for $\mu = (\mu_1,\ldots, \mu_d) \in \{0,\pm 1\}^d$ we set 
\begin{equation*}
C_{\mu} = \frac{\prod_{k=0}^{d}B_k^{\mu_{k},\mu_{k+1}}}{\prod_{k=1}^{d}b_k^{\mu_{k}}},
\end{equation*}
with the convention that $\mu_0=\mu_{d+1}=0$. 
With the above notations, we define an operator acting on the indices as follows 
\begin{align}
\CL^{\nu}_{d}:=\CL^{\nu}_{d}(\nu;\ell;N)=\sum_{0\neq \mu \in  \{-1,0,1\}^d} C_{\mu} 
\left(\prod_{k=1}^{d}E_{\nu_k}^{\mu_{k}-\mu_{k+1}}-1\right).\label{eq:operatordn}
\end{align}
Note that this operator is significantly more complicated than the difference operator $\CL^{x}_{d}$ 
in \eqref{eq:operatordx}. It can be obtained as a limit of the image of the $d$-dimensional Racah operator 
under the bispectral involution associated with the Racah polynomials, see \cite[Section 5.2]{GI}. In particular, 
this shows that
\begin{equation}\label{eq:dn}
\CL^{\nu}_{d}(\nu;\ell;N)\wh \sQ_\nu(x;\ell, N)=|x|\,\wh \sQ_\nu(x;\ell, N). 
\end{equation}

\begin{rem}[Boundary conditions]
Note that the operator $\CL^{\nu}_{d}$ contains backward shift operators, and the polynomials $\wh \sQ_\nu(x;\ell, N)$ 
are defined for $\nu\in\NN_0^d$. However, one can easily see that the coefficients of the operator $\CL^{\nu}_{d}$ 
that multiply the polynomials containing negative indices are zero, so we can simply ignore these terms. Indeed, 
the operator  $\CL^{\nu}_{d}$ will contain a negative power of $E_{\nu_k}$ in one of the following cases:
\begin{itemize}
\item[Case 1:] $\mu_{k}-\mu_{k+1}=-1$, and the corresponding term in the sum in \eqref{eq:operatordn}  
contains $E_{\nu_k}^{-1}$. Note that $\mu_{k}-\mu_{k+1}=-1$ is possible for two sets of indices
\begin{itemize}
\item  when $(\mu_k,\mu_{k+1})=(-1,0)$, in which case $C_{\mu}$ contains the term 
$B_k^{-1,0}$, which is $0$ when $\nu_k=0$,
\item or when  $(\mu_k,\mu_{k+1})=(0,1)$,  in which case $C_{\mu}$ contains the term $B_k^{0,1}$, which is  
$0$ when $\nu_k=0$.
\end{itemize}
\item[Case 2:] $\mu_{k}-\mu_{k+1}=-2$, and the corresponding term in the sum in \eqref{eq:operatordn}  contains $E_{\nu_k}^{-2}$. This happens only when $\mu_k=-1$ and $\mu_{k+1}=1$, in which case $C_{\mu}$ contains the term $B_k^{-1,1}$, which is  $0$ when $\nu_k=0$ or $\nu_{k}=1$.
\end{itemize}

This shows that we can ignore all terms containing negative indices. With this convention, the results in \cite[Section 5.2]
{GI} imply that \eqref{eq:dn} holds for all $\nu\in\NN_0^d $ when $\ell_j$ and $N$ are generic (non-integer) parameters.
For parameters $\ell_j\in\NN$ and $N\in\NN$, equation \eqref{eq:dn} holds  when the indices on the left-hand side  
belong to the set $H_{\ell,N}^d$. We will use the same convention for the other operators we construct in this section. 
\end{rem}

Similarly to the difference operators in $x$, we can construct a family of commuting difference operators acting on the indices $\nu_1,\dots,\nu_d$, which represent the multiplication by the variables $x_1,\dots,x_d$ in the basis \eqref{eq:Hahn-n}.

\begin{thm}\label{thm:rec}
If we set 
$$\CL^{\nu}_{k}:=\CL^{\nu}_{k}(\bnu_{k}; \ell_{1},\dots,\ell_{k}, |\bell^{k+1}|-2|\bnu^{k+1}|; N-|\bnu^{k+1}|), \qquad\text{ for }k=1,\dots,d-1,$$
the polynomials in \eqref{eq:Hahn-n} will satisfy also the following spectral equations 
\begin{equation}\label{eq:spn}
\CL^{\nu}_{k}\wh \sQ_\nu(x;\ell, N)= |\xb_k|\, \wh \sQ_\nu(x;\ell, N), \quad 
\text{for }k=1,\dots,d.
\end{equation}
\end{thm}

\begin{proof}
If $k=d$, equation \eqref{eq:spn} follows from \eqref{eq:dn}. 
Fix now $k<d$ and note that, for $j>k$ the terms in the product on the right-hand side of \eqref{eq:Hahn-n} are independent of $\bnu_k=(\nu_1,\dots,\nu_k)$. From this, it follows that, up to a factor independent of 
$\nu_1,\dots,\nu_k$, $\wh \sQ_\nu(x;\ell, N)$ coincides with the Hahn polynomial in the variables 
$\xt=\xb_{k}=(x_{1},\dots,x_{k})$, with indices $\tilde{\nu}=\bnu_{k}=(\nu_{1},\dots,\nu_{k})$, and parameters 
$\tilde{\ell}=(\ell_{1},\ell_{2},\dots,\ell_{k}, |\bell^{k+1}|-2|\bnu^{k+1}|)$, $\tilde{N}=N-|\bnu^{k+1}|$. 
The proof now follows from the $k$-dimensional analog of \eqref{eq:dn}.
\end{proof}

\begin{rem}
In view of \cite{DG}, we refer to  \eqref{eq:spx} and \eqref{eq:spn} as bispectral equations for the Hahn polynomials. 
\end{rem}

\begin{rem}
The equations  \eqref{eq:spn} also lead to three-term relation for $\wh \sQ_\nu(\cdot, \ell,N)$ since we can write
them as
\begin{equation}\label{eq:spn2}
(\CL^{\nu}_{k}-\CL^{\nu}_{k-1})\wh \sQ_\nu(x;\ell, N)= x_k \wh \sQ_\nu(x;\ell, N), \quad 
\text{for }k=1,\dots,d,
\end{equation}
where $\CL^{\nu}_{0}$ denotes the zero operator.  
\end{rem}

\subsection{Explicit formulas in dimension two}
In this subsection, we write the operators  and the difference equations in Theorem~\ref{thm:rec} in dimension two. Note that when $d=2$, we have $|\nu|=\nu_1+\nu_2$ and $|\ell|=\ell_1+\ell_2+\ell_3$. 
The operator in \eqref{eq:operatordn} can be written as 
\begin{align*}
\CL^{\nu}_{2}&=C_{1,0}E_{\nu_1}+C_{1,1}E_{\nu_2} +C_{1,-1}E_{\nu_1}^{2}E_{\nu_2}^{-1}\\
&\qquad +C_{0,1}E_{\nu_1}^{-1}E_{\nu_2}+C_{0,-1}E_{\nu_1}E_{\nu_2}^{-1}+\wh{C}_{0,0}\\
&\qquad\qquad +C_{-1,0}E_{\nu_1}^{-1}+C_{-1,-1}E_{\nu_2}^{-1}+C_{-1,1}E_{\nu_1}^{-2}E_{\nu_2},
\end{align*}
where
\begin{align*}
C_{1,0}&=\frac{(N-|\nu|) (\ell_1-\nu_1) (\nu_1+2\nu_2-|\ell|-1) }{(2|\nu|-|\ell|)(2|\nu|-|\ell|-1)(2\nu_2-\ell_2-\ell_3)(2\nu_2-\ell_2-\ell_3-2)}\\
&\qquad\qquad \times (2\nu_2(\nu_2-\ell_2-\ell_3-1)+\ell_3(\ell_2+\ell_3+2))\\
C_{-1,0}&=-\frac{(|\nu|-|\ell|+N-1)\nu_1(\nu_1+2\nu_2-\ell_2-\ell_3-1) }{(2|\nu|-|\ell|-2)(2|\nu|-|\ell|-1)(2\nu_2-\ell_2-\ell_3)(2\nu_2-\ell_2-\ell_3-2)}\\
&\qquad\qquad \times (2\nu_2(\nu_2-\ell_2-\ell_3-1)+\ell_3(\ell_2+\ell_3+2))\\
C_{0,1}&=\frac{(2N-|\ell|)\nu_1 (\nu_1+2\nu_2-|\ell|-1) (\ell_2-\nu_2)(\nu_2-\ell_2-\ell_3-1)}{(2|\nu|-|\ell|)(2|\nu|-|\ell|-2)(2\nu_2-\ell_2-\ell_3)(2\nu_2-\ell_2-\ell_3-1)}\\
C_{0,-1}&=\frac{(2N-|\ell|)(\ell_1-\nu_1)(\nu_1+2\nu_2-\ell_2-\ell_3-1) \nu_2(\nu_2-\ell_3-1)}{(2|\nu|-|\ell|)(2|\nu|-|\ell|-2)(2\nu_2-\ell_2-\ell_3-2)(2\nu_2-\ell_2-\ell_3-1)}\\
C_{1,1}&=\frac{(N-|\nu|) (\nu_1+2\nu_2-|\ell|-1) (\nu_1+2\nu_2-|\ell|) (\ell_2-\nu_2)(\nu_2-\ell_2-\ell_3-1)}{(2|\nu|-|\ell|)(2|\nu|-|\ell|-1)(2\nu_2-\ell_2-\ell_3)(2\nu_2-\ell_2-\ell_3-1)}\\
C_{-1,1}&=\frac{(|\nu|-|\ell|+N-1) \nu_1 (\nu_1-1) (\ell_2-\nu_2)(\nu_2-\ell_2-\ell_3-1)}{(2|\nu|-|\ell|-2)(2|\nu|-|\ell|-1)(2\nu_2-\ell_2-\ell_3)(2\nu_2-\ell_2-\ell_3-1)}\\
C_{1,-1}&=-\frac{(N-|\nu|) (\ell_1-\nu_1) (\ell_1-\nu_1-1) \nu_2(\nu_2-\ell_3-1)}{(2|\nu|-|\ell|)(2|\nu|-|\ell|-1)(2\nu_2-\ell_2-\ell_3-2)(2\nu_2-\ell_2-\ell_3-1)}\\
C_{-1,-1}&=-\frac{(|\nu|-|\ell|+N-1) (\nu_1+2\nu_2-\ell_2-\ell_3-1) (\nu_1+2\nu_2-\ell_2-\ell_3-2) }{(2|\nu|-|\ell|-2)(2|\nu|-|\ell|-1)(2\nu_2-\ell_2-\ell_3)(2\nu_2-\ell_2-\ell_3-1)}\\
&\qquad\qquad\times \nu_2(\nu_2-\ell_3-1),
\end{align*}
and  
$$\wh{C}_{0,0}=-\sum_{\mu \in  \{-1,0,1\}^2\setminus\{(0,0)\}} C_{\mu}.$$
The operator $\CL^{\nu}_{1}$ can be written as 
\begin{align*}
\CL^{\nu}_{1}&=C_{1}' (E_{\nu_1}-1) +C_{-1}' (E_{\nu_1}^{-1}-1),
\end{align*}
where
\begin{align*}
C_{1}'&=\frac{(N-|\nu|) (\ell_1-\nu_1) (\nu_1+2\nu_2-|\ell|-1) }{(2|\nu|-|\ell|)(2|\nu|-|\ell|-1)}\\
C_{-1}'&=-\frac{(|\nu|-|\ell|+N-1)\nu_1(\nu_1+2\nu_2-\ell_2-\ell_3-1) }{(2|\nu|-|\ell|-2)(2|\nu|-|\ell|-1)}.
\end{align*}
With the above formulas, equations \eqref{eq:spn} take the form
\begin{align*}
&\CL^{\nu}_{1}\wh \sQ_\nu(x;\ell, N)= x_1 \wh \sQ_\nu(x;\ell, N),\\
&\CL^{\nu}_{2}\wh \sQ_\nu(x;\ell, N)= (x_1+x_2) \wh \sQ_\nu(x;\ell, N).
\end{align*}

\end{document}